\documentclass[a4paper,10pt,leqno]{amsart}
\usepackage{amsmath,amsfonts,amsthm,amssymb}
\usepackage[alphabetic]{amsrefs}
\usepackage[OT4]{fontenc}
\usepackage{mathrsfs}
\usepackage{hyperref}
\usepackage{enumerate, xspace}
\usepackage[pdftex]{graphicx}
\usepackage{color}

\newtheorem{thm}{Theorem}[section]
\newtheorem{lem}[thm]{Lemma}
\newtheorem{prop}[thm]{Proposition}
\newtheorem{cor}[thm]{Corollary}
\newtheorem{conj}[thm]{Conjecture}
\newtheorem{claim}{Claim}
\theoremstyle{definition}

\newtheorem{defin}[thm]{Definition}

\newcommand{\eps}{\varepsilon}
\newcommand{\Z}{\mathbb{Z}}

\newcommand{\R}{\mathbb{R}}

\newcommand{\lk}{\mathrm{lk}}

\newcommand{\kol}{\color{black}}
\newcommand{\koll}{\color{black}}

\newcommand{\mc}[1]{\mathcal{#1}}
\newcommand{\bb}[1]{\mathbb{#1}}

\newcommand{\bd}[1]{\partial #1}

\begin{document}

\title[Torsion groups do not act on $2$-dimensional $\mathrm{CAT}(0)$~complexes]{Torsion groups do not act on $2$-dimensional $\mathrm{CAT}(0)$ complexes}

\author[S.~Norin]{Sergey Norin$^{*}$}
\address{
Department of Mathematics and Statistics,
McGill University,
Burnside Hall,
805 Sherbrooke Street West,
Montreal, QC,
H3A 0B9, Canada}
\email{sergey.norin@mcgill.ca}
\thanks{$*$ Partially supported by NSERC}

\author[D.~Osajda]{Damian Osajda$^{\dag}$}
\address{Instytut Matematyczny,
	Uniwersytet Wroc\l awski\\
	pl.\ Grun\-wal\-dzki 2/4,
	50--384 Wroc\-{\l}aw, Poland}
\address{Institute of Mathematics, Polish Academy of Sciences\\
	\'Sniadeckich 8, 00-656 War\-sza\-wa, Poland}
\email{dosaj@math.uni.wroc.pl}
\thanks{$\dag$ Partially supported by (Polish) Narodowe Centrum Nauki, UMO-2017/25/B/ST1/01335.}

\author[P.~Przytycki]{Piotr Przytycki$^{\ddag}$}
\address{
Department of Mathematics and Statistics,
McGill University,
Burnside Hall,
805 Sherbrooke Street West,
Montreal, QC,
H3A 0B9, Canada}

\email{piotr.przytycki@mcgill.ca}
\thanks{$\ddag$ Partially supported by Narodowe Centrum Nauki UMO-2015/18/M/ST1/00050, NSERC, and FRQNT}

\begin{abstract}
	
	We show, under mild hypotheses, that if each element of a finitely generated group acting on a $2$-dimensional $\mathrm{CAT}(0)$ complex has a fixed point, then there is a global fixed point. In particular all actions of finitely generated torsion groups on such complexes have global fixed points.
	The proofs rely on Masur's theorem on periodic trajectories in rational billiards, and Ballmann--Brin's
	methods for finding closed geodesics in $2$-dimensional locally $\mathrm{CAT}(0)$ complexes.
    As another ingredient we prove that the image of an immersed loop in a graph of girth $2\pi$ with length not commensurable with $\pi$ has diameter $>\pi$.
This is closely related to a theorem of Dehn on tiling rectangles by squares.
\end{abstract}

\maketitle

\section{Introduction}

Let $H$ be a group acting properly and cocompactly on
a $2$-dimensional CAT(0) complex $X$. In \cite{BB} Ballmann and Brin proved \emph{Rank Rigidity} for $H$, saying that if each edge of $X$ is in at least two 2-cells, then either $H$ has an element of rank~$1$, or~$X$ is a Euclidean building. Two other strong\-ly related questions on CAT(0) groups have remained open even in the same $2$-dimensional setting. The first one is the \emph{Tits Alternative} stating that subgroups of $H$ are either virtually abelian or contain free subgroups. The second one
is even more basic (as it might be seen as a first step for proving the Tits Alternative):

\begin{center}
	\textbf{Question.} \emph{Can $H$ have an {infinite torsion subgroup} $G$?}
\end{center}
(See e.g.\ \cite{Swe}, \cite[Quest 2.11]{BestvinaProblemOLD}, \cite[Quest 8.2]{BridsonProblem}, \cite[Prob 24]{KapovichProblems}, \cite[\S~IV.5]{Cap} for appearances of the problem.)
Surprisingly, prior to our work the answer was not known even in the otherwise very well understood case of lattices
in isometry groups of Euclidean buildings of type $\widetilde A_2$. Swenson \cite{Swe} proved that for $1$-ended $H$, a negative answer to the Question implies the non-existence of cut points in the CAT(0) boundary of $X$. (Later, Papasoglu and Swenson \cite{PapSwe} showed the non-existence of boundary cut points independent of the answer to the Question.) Moreover, Swenson answered the Question in the negative for $G=H$.
Recently, Papasoglu and Swenson \cite{PapSwe2} proved that if $G<H$ is an infinite torsion subgroup, then $G$ does not fix a point in the limit set $\Lambda G$.
(The above results are independent of the dimension of the CAT(0) space.)
Furthermore, if $G$ is finitely presented, then it also acts properly and cocompactly on a $2$-dimensional CAT(0) complex \cite[Thm 1.1]{HMP}, so $G$ is not infinite torsion.

\medskip

Generalising the Question by discarding $H$ one can ask if there is an infinite torsion group $G$ acting without a global fixed point on a $2$-dimensional CAT(0) complex.
It is natural to ask this question for $G$ finitely generated, since infinitely generated countable torsion groups can act properly on trees (see \cite[II.7.11]{BH}).
We answer this generalisation of the Question, as well as the Question itself, in the negative in Corollaries~\ref{cor:proper} and~\ref{cor:subgroup}.
These are two consequences of our main Theorem~\ref{thm:torsion}, concerning
more general actions of groups on $2$-dimensional CAT(0) complexes. Under mild hypotheses, it states
that locally elliptic actions have global fixed points. This result as well as both corollaries are new
even in the case of $2$-dimensional Euclidean buildings. In particular, they confirm the corresponding special case of \cite[Conj~1.2]{Mar}, and extend results of \cite{Parreau}. Note that finite-dimensionality assumption is important, since there are infinite torsion Grigorchuk groups that are amenable \cite{Gri}, and hence they act properly on a Hilbert space, which is CAT(0). Also Burnside groups, that is, infinite torsion groups of bounded exponent can act without a global fixed point on infinite dimensional CAT(0) cubical complexes \cite{Osa}.

\medskip

We now pass to discussing the general setup for Theorem~\ref{thm:torsion}.
Let $X$ be a $2$-dimensional simplicial complex, which we will call a \emph{triangle complex}.
We assume additionally that $X$ has a \emph{piecewise smooth Riemannian metric}, which is a family of smooth Riemannian metrics $\sigma_T, \sigma_e$ on the triangles and edges such that the restriction of $\sigma_T$ to $e$ is $\sigma_e$ for each $e\subset T$. Riemannian metrics $\sigma_T, \sigma_e$ induce metrics (i.e.\ distance functions) $d_T,d_e$ on triangles and edges. We then equip $X$ with the \emph{quotient pseudometric}~$d$ (see \cite[I.5.19]{BH}).

We assume that the triangles containing each vertex $v$ of $X$ belong to only finitely many isometry classes of $\sigma_T$.
We also assume that $(X,d)$ is a complete length space. This holds for example:
\begin{itemize}
\item
if there are only finitely many isometry classes of $\sigma_T$ (see \cite[I.7.13]{BH} for $X$ an \emph{$M_\kappa$ complex}, meaning that each $\sigma_T$ has Gaussian curvature $\kappa$ and geodesic sides; the general case follows using a bilipschitz map from $X$ to an $M_\kappa$ complex), or
\item
if $X$ is the space $\bf X$ for the tame automorphism group $\mathrm{Tame}(\mathbf k^3)$, with $\mathbf k$ of characteristic $0$ \cite[Prop 5.4 and~Lem 5.8]{LP}. Some cells of $\bf X$ are polygons instead of triangles, but we can easily transform $\bf X$ into a triangle complex by subdividing.
\end{itemize}

We consider an action of a group $G$ on $X$ by \emph{automorphisms}, i.e.\ simplicial automorphisms that preserve the metrics $\sigma_T$. Consequently, they are isometries of $(X,d)$.

See Section~\ref{sec:spaces} for the discussion of the CAT(0) property and the definition of having \emph{rational angles}. For example, if the
triangles of $X$ have all angles commensurable with $\pi$, then $X$ has rational angles.

\begin{thm}
\label{thm:torsion}
Let $(X,d)$ be a $\mathrm{CAT}(0)$ triangle complex.
Let $G$ be a finitely generated group acting on $X$ without a global fixed point.
Assume that
\begin{enumerate}[(i)]
\item each element of $G$ fixing a point of $X$ has finite order, or
\item $X$ is locally finite, or
\item $X$ has rational angles {\kol with respect to $G$.}
\end{enumerate}
Then $G$ has an
element with no fixed point in $X$.
\end{thm}

While we believe that Theorem~\ref{thm:torsion} holds also without (i),(ii), and~(iii), this does not seem to be tractable with our methods. For
example, if we wanted to keep the same line of argument as under condition~(iii), we would need in particular the existence of periodic trajectories
in triangular billiards, which is a major open problem (see e.g.\ \cite{M}).

Note that we do not assume that the edges of a triangle $T$ of $X$ have geodesic curvature $0$ in $\sigma_T$ (see $\bf X$ above for an example). On
the other hand, if we made this additional assumption, we could replace each $\sigma_T$ by the Euclidean metric, which would allow us to discard
Lemma~\ref{lem:new} and the last three paragraphs of the proof of Proposition~\ref{prop:loxodromic}, and which would simplify the proof of
Lemma~\ref{lem:link}.

Applying Theorem~\ref{thm:torsion} and then \cite[Thm~1.1]{CL} to the family of fixed point sets for all finitely generated subgroups $G$ of $H$
below gives:

\begin{cor}
\label{cor:bdry}
Let $(X,d)$ be a $\mathrm{CAT}(0)$ triangle complex.
Let $H$ be a group acting on $X$ such that each element of $H$ fixes a point of $X$.
Assume that (i),(ii), or (iii) holds, {\kol with respect to $H$}. Then $H$ fixes a point of $X\cup \partial X$.
\end{cor}

Theorem~\ref{thm:torsion} under condition~(i) can be rephrased in the following way.

\begin{cor}
\label{cor:proper}
Let $(X,d)$ be a $\mathrm{CAT}(0)$ triangle complex. Let $G$ be a finitely generated group acting on $X$ without a global fixed point. Then $G$ has an element of infinite order.
\end{cor}

The following corollary answers the Question. Note that the finite generation assumption is removed here.

\begin{cor}
\label{cor:subgroup}
Let $(X,d)$ be a $\mathrm{CAT}(0)$ triangle complex.
Let $H$ be a group acting properly and cocompactly on $X$. Then any subgroup $G$ of $H$ is finite or has an element of infinite order.
\end{cor}

\begin{proof}
Suppose that all elements of $G$ have finite order. By \cite[II.2.8(2)]{BH}, there is a finite bound on the size of all finite subgroups of $H$. Let then $F$ be a maximal finite subgroup of $G$. We will prove that $F=G$. Otherwise, there is $g\in G-F$. Then $\langle F, g\rangle$ is finitely generated and thus by Corollary~\ref{cor:proper} it is finite. This contradicts the maximality of $F$.
\end{proof}

In view of the Question and Corollary~\ref{cor:proper} we state the following conjecture, which we could not
find elsewhere in the literature. Observe that, as in the proof of Corollary~\ref{cor:subgroup}, the conjecture implies a negative answer to the Question in the setting of finite-dimensional CAT(0) complexes.

\begin{conj}
	\label{conj:torsion}
	Every finitely generated group acting without a global fixed point on a finite-dimen\-sio\-nal $\mathrm{CAT}(0)$ complex contains an element
	of infinite order.
\end{conj}

Note that the conjecture is known to hold for $\mathrm{CAT}(0)$ cubical complexes by a work of Sageev \cite[proof of Thm 5.1]{Sag}. The proof relies heavily on the structure of the space with walls for such complexes.

In the proof of Theorem~\ref{thm:torsion} under conditions (i) or (ii) we will need the following graph-theoretic result of independent interest.

Below, by a \emph{graph} we mean a (possibly infinite) metric graph with finitely many possible edge lengths.
A closed edge-path embedded in a graph $\Lambda$ is a \emph{cycle of $\Lambda$}. An embedded edge-path $P$ in $\Lambda$ is a \emph{segment of $\Lambda$} if the endpoints of $P$ have degree (i.e.\ valence) at least three in $\Lambda$, but every internal vertex of $P$ has degree two.

\begin{thm}\label{thm:graph+}
Let $\Gamma$ be a graph, and let $\Gamma'$ be a finite subgraph of $\Gamma$ with all vertices of degree at least two.
Assume that
\begin{itemize}
\item the girth of $\Gamma$ is $\geq 2\pi$, and
\item each $x,y$ in $\Gamma'$ are at distance $\leq \pi$ in $\Gamma$.
\end{itemize}
Then all cycles and segments of $\Gamma'$ have length commensurable with~$\pi$.
\end{thm}

The special case of Theorem~\ref{thm:graph+}, where $\Gamma'=\Gamma$, was established by Ballmann and Brin \cite[Lem~6.1]{BB}, who proved additionally that such $\Gamma$ is a spherical building, and so in particular all its segments have the same length.

Theorem~\ref{thm:graph+} has the following immediate consequence.

\begin{cor}
\label{thm:graph}
Let $\Gamma$ be a graph, and let $\gamma\colon C\to \Gamma$ be a closed edge-path immersed in $\Gamma$. Assume that
\begin{itemize}
\item the girth of $\Gamma$ is $\geq 2\pi$, and
\item each $x,y$ in $\gamma(C)$ are at distance $\leq \pi$ in $\Gamma$.
\end{itemize}
Then the length of $\gamma$ is commensurable with $\pi$.
\end{cor}

Interestingly, Theorem~\ref{thm:graph+} will follow from a generalisation of a classical theorem of Dehn~\cite{Dehn} stating that a rectangle tiled by finitely many squares has commensurable side lengths.

\medskip

\noindent \textbf{Idea of proof of Theorem~\ref{thm:torsion}.} Suppose by contradiction that for each $f\in G$ the set $\mathrm{Fix}(f)$ of fixed points of $f$ is nonempty. We wish to prove inductively that for any finite set of elements $f_1,\ldots, f_n\in G$ (in particular for a generating set) the intersection $\mathrm{Fix}(f_1)\cap \cdots \cap \mathrm{Fix}(f_n)$ is nonempty. Here we illustrate the induction step from $n=2$ to $n=3$. Suppose that $\mathrm{Fix}(f_1),\mathrm{Fix}(f_2),\mathrm{Fix}(f_3)$ pairwise intersect but their triple intersection is empty. We then find a simplicial disc $\Delta$ with decomposition of its boundary into three paths $\partial \Delta=P\cup Q\cup R$ together with a simplicial map $\psi$ from $\Delta$ to $X$ that sends $P,Q,R$ into $\mathrm{Fix}(f),\mathrm{Fix}(g),\mathrm{Fix}(h),$ respectively. We suppose that~$\Delta$ has minimal possible area. Under condition~(iii) on rational angles we find a periodic billiard trajectory $\omega$ in~$\Delta$ outside its $0$-skeleton~$\Sigma$. Then $\psi(\omega)$ develops to an axis in $X$ of a loxodromic element of $G$ and leads to contradiction. If the angle at a vertex $v\in \Sigma$ is not commensurable with $\pi$, then developing the image under $\psi$ of the link at $v$ and using condition~(i) or~(ii) we construct
a closed edge-path in the link of $\psi(v)$ in~$X$ of length not commensurable with $\pi$. Applying Corollary~\ref{thm:graph} to that path we obtain directions at $\psi(v)$ at distance $>\pi$. These directions permit to construct a trajectory through $v$ that develops to an axis of a loxodromic element.

\medskip

\noindent \textbf{Organisation.} In Section~\ref{sec:spaces} we discuss CAT(0) spaces and the rational angles property. In Section~\ref{sec:2} we outline the proof of Theorem~\ref{thm:torsion} for $2$-generated $G$. In Sections~\ref{sec:surface} and~\ref{sec:loxodromic} we fill in the details of that outline. In Section~\ref{sec:3} we complete the proof of Theorem~\ref{thm:torsion}. The proof of Theorem~\ref{thm:graph+} is postponed till Section~\ref{sec:graph}.

\medskip

\noindent \textbf{Acknowledgements.} We thank Werner Ballmann, Martin Bridson, Pierre-Emma\-nuel Caprace, Jingyin Huang, St\'ephane Lamy, Timoth\'ee
Marquis, Fr\'ed\'eric Paulin, and Eric Swenson for introducing us to the state of affairs in the subject, for their comments and encouragement. We
thank the referees for corrections and remarks. This paper was written while D.O.\ was visiting McGill University. We would like to thank the
Department of Mathematics and Statistics of McGill University for its hospitality during that stay.

\section{Rational angles}
\label{sec:spaces}

{\koll Let $v$ be a vertex of a triangle complex $X$.
Let $\lk_v$ be the graph that is the \emph{link} of~$v$, as defined in \cite[page~176]{BB}. Namely, vertices of $\lk_v$ correspond to edges of~$X$ containing $v$ and edges of $\lk_v$ correspond to triangles of $X$ containing $v$.}
We treat $\lk_v$ as a length metric space $(\lk_v,d_v)$ where the length of each edge is the angle in an appropriate triangle of $X$. Since we assumed that triangles containing~$v$ belong to only finitely many isometry classes of $\sigma_T$, there are only finitely many possible edge lengths in a given $\lk_v$, and so $\lk_v$ is complete.

For a piecewise smooth simplicial map $\phi \colon X\to X'$ we keep the notation $\phi$ for all the maps $\lk_v\to \lk_{\phi(v)}$ induced by $\phi$.

The \emph{space of directions} $S_v$ at $v$ is the set of geodesics issuing from $v$, where we identify geodesics at Alexandrov angle~$0$ (see
\cite[II.3.18]{BH}). The Alexandrov angle equips $S_v$ with a metric.

Suppose that $X$ is $\mathrm{CAT}(0)$. Consider the map $i \colon \lk_v\setminus \lk_v^0\to S_v$ mapping each point in the interior of an edge of
$\lk_v$ to the class of the appropriate geodesic in the corresponding triangle of $X$ containing~$v$ (which remains a geodesic in $X$ since $X$ is
$\mathrm{CAT}(0)$). Note that we cannot similarly define~$i$ on the vertex set $\lk_v^0$, since for example for $X=T$ a single triangle with
$(T,\sigma_T)$ isometric to a sector of a Euclidean disc, and $v\neq w$ vertices of $T$ distinct from the centre of the disc, there is no geodesic
issuing from $v$ tangent to the edge $vw$ of $T$.

\begin{lem}
\label{lem:link} Let $v$ be a vertex of a $\mathrm{CAT}(0)$ triangle complex $X$. Suppose that each edge of~$X$ containing~$v$ is contained in a
triangle of $X$. Then $i$ extends to an isometry from~$\lk_v$ considered with the metric that is the minimum of $d_v$ and $\pi$ to the
completion~$\overline{S}_v$ of~$S_v$.
\end{lem}

\begin{proof}
We first prove that the image of $i$ is dense in~$S_v$. Indeed, after possibly subdividing $X$, we can assume that in each triangle $T$ containing
$v$, the geodesic~$\gamma_T$ starting at~$v$ and bisecting the angle of $T$ at $v$ ends at the opposite side of $T$. Since~$X$ is uniquely geodesic,
a geodesic $vx$ cannot intersect $\gamma_T$ transversely. Thus if~$vx$ is contained in the union $\mathrm{St}(v)$ of all the closed triangles
containing $v$, then $vx$ cannot intersect outside $v$ two distinct edges containing $v$. Consequently, we have at least one of the following
situations. If $vx$ contains a geodesic $vy$ contained in a triangle of $X$, then the class of $vx$ in $S_v$ lies in the closure of the image under
$i$ of the corresponding open edge of~$\lk_v$. If each $vy\subset vx$ intersects infinitely many times the same edge $e$ containing~$v$, then the
class of $vx$ in $S_v$ lies in the closure of the image under~$i$ of any open edge of $\lk_v$ corresponding to a triangle containing $e$.

Since $d_v$ is a complete metric on $\lk_v$, the minimum of $d_v$ and $\pi$ is also complete. Thus to prove the lemma it remains to justify that $i$
is an isometric embedding. This is easy in the special case where all the edges of $X$ containing $v$ are geodesics, and we reduce the general case
to that special case in the following way. It is easy to see that $i$ is 1-Lipschitz. Conversely, let $\xi,\eta\in \lk_v$ be at distance $>\eps>0$
from~$\lk_v^0$, with $2\eps$ smaller than the length of the shortest edge of $\lk_v$. After possibly subdividing $X$, we can assume that in each
triangle $T$ containing $v$, the geodesics $\gamma^e_T,\gamma^f_T$ starting at $v$ at angle $\eps$ from the edges $e,f$ of $T$ end at the opposite
side of~$T$, and that for given $e$, various $\gamma^e_T$ have the same length. Let $\mathrm{St}_\eps(v)$ be the triangle complex obtained from
$\mathrm{St}(v)$ by replacing each triangle $T$ by the triangle $T_\eps\subset T$ bounded by $\gamma^e_T\cup \gamma^f_T$ and an arc in the side of
$T$ opposite to $v$, and by identifying all $\gamma^e_T$ for given~$e$. Let $p_\eps\colon \mathrm{St}(v)\to\mathrm{St}_\eps(v)$ be the $1$-Lipschitz
map whose restriction to each~$T_\eps$ is the identity, and which maps each point $x$ in the component of $\mathrm{St}(v)\setminus \bigcup T_\eps$
containing~$e\setminus v$ to the point on $\gamma^e_T$ at distance $\min\{d(x,v),|\gamma^e_T|\}$ from~$v$. Let $\xi'$ and~$\eta'$ be geodesics in
$\mathrm{St}(v)$ that represent $i(\xi)$ and $i(\eta)$. It is clear that the Alexandrov angle between $p_\eps(\xi')$ and $p_\eps(\eta')$ does not
exceed the Alexandrov angle between $\xi'$ and $\eta'$. Furthermore, by the special case above applied to $\mathrm{St}_\eps(v)$, the former
Alexandrov angle converges to $\min\{d_v(\xi,\eta),\pi\}$ as $\eps\to 0$. This justifies that $i$ is an isometric embedding.
\end{proof}

For $X$ to be $\mathrm{CAT}(0)$ it is necessary that
\begin{enumerate}[(a)]
\item the girth of each $\lk_v$ is $\geq 2\pi$,
\item the Gaussian curvature of $\sigma_T$ at any interior point of $T$ is $\leq 0$,
\item the sum of geodesic curvatures in any two distinct triangles at any interior point of a common edge is $\leq 0$, and
\item $X$ is simply connected.
\end{enumerate}

Note that condition (c) can be justified exactly as in the proof of \cite[Thm~7.1]{Babu}. Condition (a) follows from Lemma~\ref{lem:link}, since we
can assume without loss of generality that each edge of $X$ containing~$v$ is contained in a triangle of $X$, and $\overline{S}_v$ is
$\mathrm{CAT}(1)$ by \cite[II.3.19]{BH}.

While it seems that (a)--(d) are also sufficient conditions for $X$ to be $\mathrm{CAT}(0)$ (one would need to couple the ideas from \cite{Babu} and
\cite{BH}), only under the following assumptions this has been verified in the literature.

\begin{itemize}
\item $X$ is locally finite \cite[Thm 7.1]{Babu}, or
\item $X$ is an $M_\kappa$ complex with $\kappa\leq 0$ and there are only finitely many isometry classes of $\sigma_T$ \cite[II.5.2 and~II.5.4]{BH}, or
\item $X$ is the space $\bf{X}$ for the tame automorphism group $\mathrm{Tame}(\mathbf k^3)$ \cite[Thm~A]{LP}.
\end{itemize}

The following lemma guarantees the existence of a sensible \emph{barycentric subdivision} for triangle complexes.

\begin{lem}
\label{lem:bar}
Let $X$ be a triangle complex. Then we can equip the barycentric subdivision $X'$ of $X$ with a piecewise smooth Riemannian metric so that we have an isometry $X\to X'$ under which each automorphism of $X$ is mapped to an automorphism of $X'$.
\end{lem}
\begin{proof}
To obtain $X'$, we subdivide each edge of $X$ into two edges of equal length. Now, let $(T,\sigma_T)$ be an isometry class of a triangle of $X$. Let $\mathrm{Aut}(T)$ be the automorphism group of $T$. Our goal is to pick an interior point $m$ of $T$ and a collection of six disjoint smooth arcs in the interior of $T$ joining $m$ with the vertices and edge midpoints of $T$ that is invariant under $\mathrm{Aut}(T)$.

If $\mathrm{Aut}(T)$ is trivial, then we can choose $m$ and the arcs arbitrarily. If $\mathrm{Aut}(T)$ contains only one nontrivial element $f$ that fixes a vertex $v$ of $T$, then $\mathrm{Fix}(f)$ is a smooth arc from $v$ to the midpoint of the opposite edge. We then pick $m\in \mathrm{Fix}(f)$, two of the arcs of our collection inside $\mathrm{Fix}(f)$, two others in one component of $T-\mathrm{Fix}(f)$, and two last ones in the other component, as the images under $f$ of the preceding ones. If $\mathrm{Aut}(T)$ is the dihedral group, with involutions $f,g,h,$ then we choose~$m$ to be the unique intersection point of $\mathrm{Fix}(f),\mathrm{Fix}(g), \mathrm{Fix}(h)$, and choose the arcs inside these fixed point sets. Finally, if $\mathrm{Aut}(T)=\langle f\rangle$ is cyclic of order $3$, then let $m$ be the unique fixed point of $f$. Inside the quotient orbifold $T/\mathrm{Aut}(T)$ choose two disjoint smooth arcs joining $m$ to the unique vertex and edge midpoint, and lift them to $T$.
\end{proof}

\begin{defin}
\label{def:rational}
Let $v$ be a vertex of $X$. An immersed edge-path $\gamma$ in $\lk_v$ is \emph{flat} if
\begin{enumerate}[(i)]
\item
each edge of $\gamma$ corresponds to a triangle $T$ of $X$ with Gaussian curvature $0$ at any point of $T$, and
\item
the sum of geodesic curvatures in any two such consecutive triangles $T,T'$ at any point of $T\cap T'$ is $0.$
\end{enumerate}
We say that $X$ has \emph{rational angles} with respect to an action of a group $G$, if for each vertex $v$ of $X$ we have a discrete set $\Lambda\subset \lk_v$ that is
\begin{itemize}
\item
invariant under the stabiliser~$G_v$, and such that
\item
{\kol for any subdivision $X'$ of $X$ inducing a subdivision $\lk'_v$ of $\lk_v$ with $\Lambda$ contained in the vertex set of $\lk'_v$,
each flat immersed edge-path $\gamma$ in $\lk'_v$} that is disjoint from $\Lambda$, except possibly at the endpoints, is
\begin{itemize}
\item
finite, and
\item
if it has endpoints in $\Lambda$, then its length is commensurable with $\pi$.
\end{itemize}
\end{itemize}
\end{defin}

For example, if all triangles of $X$ have only angles commensurable with $\pi$, then $X$ has rational angles {\kol with respect to any automorphism group of $X$}, since we can take for $\Lambda$ all the vertices of $\lk_v$. This includes all $\widetilde{A}_2, \widetilde{C}_2$, and $\widetilde{G}_2$ buildings.

{\kol Note that if a triangle complex $X$ has rational angles with respect to an action of a group $G$, then $X$ has rational angles with respect to the action of any subgroup of~$G$.}

\begin{lem}
\label{lem:barycentric_preserves_rational}
Let $X$ be a triangle complex that has rational angles with respect to an action of a group $G$. Then its barycentric subdivision $X'$ has also rational angles with respect to $G$.
\end{lem}
\begin{proof}
Let $v$ be a vertex of $X'$. We will denote by $\lk'_v$ the {\koll link of}  $v$ in $X'$. Our goal is to find $\Lambda'\subset \lk'_v$ satisfying Definition~\ref{def:rational}.

If $v$ is the barycentre of a triangle $T$ of $X$, then $\lk'_v$ is a cycle of combinatorial length~$6$ and angular length~$2\pi$. Let $\lambda$ be a vertex of $\lk'_v$. In the case where $\mathrm{Aut}(T)$ has order $2$, we additionally require that $\lambda$ is fixed by $\mathrm{Aut}(T)$. Then we can take $\Lambda'=\mathrm{Aut}(T)\lambda$. If $v$ is the midpoint of an edge $e$ of {\kol $X$}, then we can take $\Lambda'$ to consist of the two vertices of $\lk'_v$ corresponding to $e$.

Finally, assume that $v$ is also a vertex of $X$. Then $\lk'_v$ is a subdivision of $\lk_v,$ {\koll the link of $v$ in $X$}. Let $\Lambda\subset \lk_v$ be the set satisfying Definition~\ref{def:rational}. Then {\kol it suffices to take} $\Lambda'=\Lambda$ under the identification of $\lk'_v$ with $\lk_v$.
\end{proof}

\begin{lem}
\label{lem:Lamy}
The space $\bf X$ for $\mathrm{Tame}(\mathbf k^3)$ constructed in \cite{LP} has rational angles.
\end{lem}
\begin{proof}
The proof is aimed at readers familiar with \cite{LP}. Let $\rho_+\colon \bf{X}\to \nabla_+$ be the folding from \cite[Cor 2.5]{LP}.
Let $P\subset \lk_{\rho_+(v)}$ be the directions determined by the \emph{principal lines}: through $\rho_+(v)$ and $[1,0,0],[0,1,0],$ or $[0,0,1]$. Let $\Lambda=\rho_+^{-1}(P)$. These are the only possible directions of edges from $\rho_+(v)$ that are geodesics, since all edges lie in admissible lines (see \cite[\S 4.A]{LP}) and the only admissible lines that are geodesics are principal (equality case in \cite[Lem~5.1]{LP}). Let $\gamma$ be a flat immersed edge-path in {\kol a subdivision of} $\lk_v$ that is disjoint from $\Lambda$, except possibly at the endpoints. Thus by Definition~\ref{def:rational}(ii), the image of $\gamma$ under $\rho_+$ in $\lk_{\rho_+(v)}$ might not be immersed only at $P$. Since consecutive points of~$P$ are at distance~$\frac{\pi}{3}$ (see \cite[Rm~5.2]{LP}), the length of~$\gamma$ is bounded by $\frac{\pi}{3}$ with equality if and only if $\gamma$ has endpoints in~$\Lambda$.
\end{proof}

Lemma~\ref{lem:Lamy} and Theorem~\ref{thm:torsion}(iii) have the following immediate consequence, which is the first step towards the Tits Alternative for $\mathrm{Tame}(\mathbf k^3)$.

\begin{cor} Let $G$ be a finitely generated subgroup of $\mathrm{Tame}(\mathbf k^3)$ acting on $\mathbf{X}$ without a global fixed point. Then $G$ has an element with no fixed point in $\mathbf{X}$.
\end{cor}

\section{Pairs of generators}
\label{sec:2}

In this section we prove Theorem~\ref{thm:torsion} for a 2-generated group $G$.
For an automorphism~$f$ of a simplicial complex $X$, $\mathrm{Fix}(f)\subset X$ denotes the set of fixed points of $f$.

\begin{prop}
\label{prop:pairs}
Let $f,g$ be automorphisms of a $\mathrm{CAT}(0)$ triangle complex $X$ satisfying condition (i),(ii), or (iii) of Theorem~\ref{thm:torsion}, {\kol with respect to $\langle f,g\rangle$}. Suppose that both $\mathrm{Fix}(f), \mathrm{Fix}(g)$ are nonempty. Then $\mathrm{Fix}(f)$ intersects $\mathrm{Fix}(g)$, or
$\langle f,g\rangle$ contains an element with no fixed point in $X$.
\end{prop}

In the proof we will need the following terminology. We say that an automorphism $f$ of a simplicial complex $X$ acts \emph{without inversions}, if whenever $f$ stabilises a simplex of $X$, then it fixes it pointwise.

\begin{defin}
\label{def:equivariant}
Let $S$ be a two-sphere with three marked points $p,q,r$, and a basepoint $z$ distinct from $p,q,r$. Choose $f^{-1}_S,g_S$ and $h_S=g^{-1}_S\circ f_S$ to be elements of the fundamental group $\pi_1(S-\{p,q,r\},z)$, represented by embedded closed paths circling counterclockwise the points $p,q,r,$ respectively, as in Figure~\ref{fig:sphere2}. Let $\widetilde S$ be the branched cover of $S$ over $p,q,r$ corresponding to the universal cover of $S-\{p,q,r\}$. For a pair $(f,g)$ of automorphisms of a $2$-dimensional simplicial complex $X$, an \emph{equivariant triangulation} of $\widetilde {S}$ \emph{w.r.t.\ $(f,g)$} is the following object.

First of all, it consists of a structure on $S$ of a \emph{$\Delta$-complex} (as defined in \cite[\S 2.1]{H}), which is a generalisation of a simplicial complex allowing the simplices not to be embedded, and allowing several simplices with the same boundary. We require that $p,q,$ and $r$ are vertices. We pull back the $\Delta$-complex structure on $S$ to a $\Delta$-complex structure on $\widetilde S$.

Secondly, an equivariant triangulation consists of a simplicial map $\phi$ from $\widetilde S$ to~$X$ that is equivariant with respect to the homomorphism from $\pi_1(S-\{p,q,r\},z)$ to $\langle f, g\rangle$ defined by $f_S\to f, g_S\to g$, which we will denote by $\phi_*$.
\end{defin}

\begin{figure}[h!]
	\begin{center}
		\includegraphics[scale=0.42]{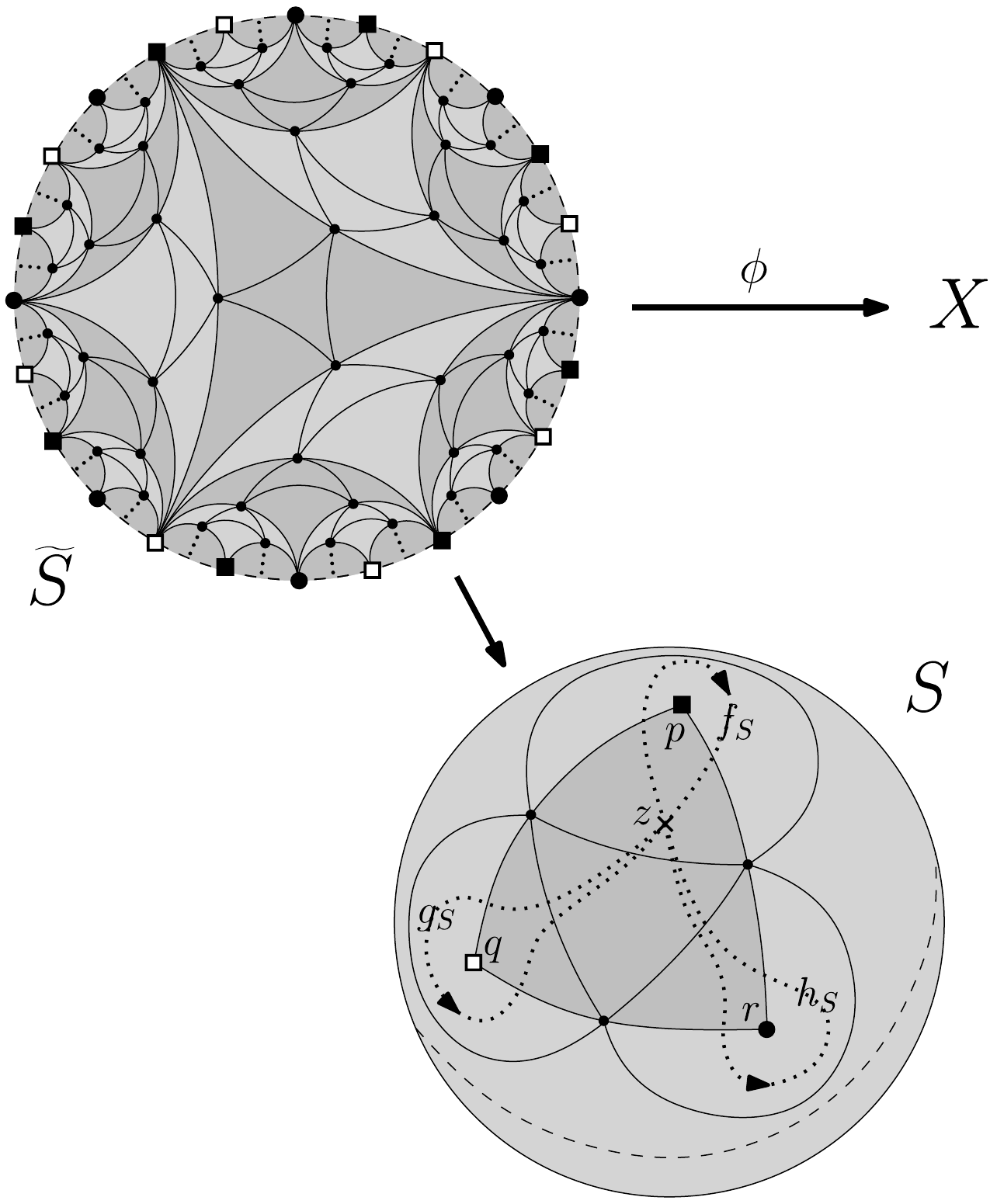}
	\end{center}
	\caption{An equivariant triangulation.}\label{fig:sphere2}
\end{figure}

A simplicial map is \emph{nondegenerate} if it does not collapse an edge to a vertex. A nondegenerate simplicial map is a \emph{near-immersion} if it is a local embedding at each edge midpoint of the source complex.

We begin with listing a lemma and a proposition that piece together to a proof of Proposition~\ref{prop:pairs}.

\begin{lem}
\label{lem:min surface}
Let $f,g,h=g^{-1}\circ f$ be automorphisms of a simply connected $2$-dimensional simplicial complex~$X$, {\kol with $\langle f,g\rangle$} acting without inversions. Suppose that all $\mathrm{Fix}(f)$, $\mathrm{Fix}(g),\mathrm{Fix}(h)$ are nonempty and pairwise disjoint. Then there exists an equivariant triangulation of $\widetilde S$ w.r.t.\ $(f,g)$ with $\phi$ a near-immersion.
\end{lem}

Lemma~\ref{lem:min surface} will be proved in Section~\ref{sec:surface}.

\begin{prop}
	\label{prop:loxodromic}
	Let $f,g$ be automorphisms of a $\mathrm{CAT}(0)$ triangle complex $X$ satisfying condition (i),(ii), or (iii) of Theorem~\ref{thm:torsion}, {\kol with respect to $\langle f,g\rangle$}. Suppose that we have an equivariant triangulation of $\widetilde S$ w.r.t.\ $(f,g)$ with $\phi$ a near-immersion. Then $\langle f,g\rangle$ contains {\kol an element with no fixed point in $X$.}
\end{prop}

Proposition~\ref{prop:loxodromic} will be proved in Section~\ref{sec:loxodromic}.

\begin{proof}[Proof of Proposition~\ref{prop:pairs}]
After possibly passing to the barycentric subdivision of $X$ (which obviously preserves the conditions~(i) and~(ii) of Theorem~\ref{thm:torsion}, and preserves condition~(iii) by Lemma~\ref{lem:barycentric_preserves_rational}), we can assume that {\kol $\langle f,g\rangle$} acts without inversions. Let $h=g^{-1}\circ f$.
If $h$ does not have a fixed point in $X$, then the proof is complete, so we can assume that $\mathrm{Fix}(h)$ is nonempty.
Similarly, we can suppose that $\mathrm{Fix}(f)$ is disjoint from $\mathrm{Fix}(g)$.
Then $\mathrm{Fix}(h)$ is disjoint from both $\mathrm{Fix}(f)$ and $\mathrm{Fix}(g)$. By Lemma~\ref{lem:min surface}, there exists an equivariant triangulation of $\widetilde S$ w.r.t.\ $(f,g)$ with $\phi$ a near-immersion. By Proposition~\ref{prop:loxodromic}, $\langle f,g\rangle$ contains {\kol an element with no fixed point in $X$.}
\end{proof}

\section{Existence of nearly immersed triangulations}
\label{sec:surface}

The \emph{area} of a finite $2$-dimensional $\Delta$-complex is the number of its triangles. We say that an equivariant triangulation of $\widetilde S$ w.r.t.\ $(f,g)$ has \emph{minimal area} if the corresponding $\Delta$-complex structure on $S$ has minimal area among all equivariant triangulations of $\widetilde S$ w.r.t.\ $(f,g)$.

We will need the following relative simplicial approximation theorem of Zeeman.

\begin{thm}[{\cite{Z}}]
\label{thm:Z}
Let $D_0$ and $X$ be finite simplicial complexes, and let $L$ be a subcomplex of $D_0$. Let $\phi\colon D_0 \to X$  be a continuous map such that the restriction of $\phi$ to $L$ is simplicial. Then there are a simplicial subdivision $D$ of $D_0$ such that $L$ remains a subcomplex of $D$, and a simplicial map $\psi \colon D\to  X$ such that the restrictions of $\psi$ and $\phi$ to $L$ coincide.
\end{thm}

\begin{proof}[Proof of Lemma~\ref{lem:min surface}]

\smallskip
\noindent \textbf{Step 1.} \emph{There exists an equivariant triangulation.}
\smallskip

Since $f,g$ act without inversions, $\mathrm{Fix}(f),\mathrm{Fix}(g),\mathrm{Fix}(h)$ are subcomplexes. Consider any vertices $a\in \mathrm{Fix}(f), b\in \mathrm{Fix}(g),c\in \mathrm{Fix}(h),$ and any nontrivial edge-paths $\alpha$ (resp.\ $\beta$) from $c$ to $a$ (resp.\ from $c$ to $b$) in $X$. Let $\gamma\colon L\to X$ be the closed edge-path $\alpha^{-1}\beta g(\beta^{-1})f(\alpha)$, which passes through $a,c,b,g(c)=f(c)$, see the top right of Figure~\ref{fig:sphere}. Let $D_0$ be the disc with the structure of a simplicial complex obtained by conning off the cycle $L$. Since $X$ is simply connected, $\gamma$ extends to a continuous map $\phi\colon D_0 \to X$. By Theorem~\ref{thm:Z}, applied with $X$ replaced by the finite subcomplex of $X$ containing the image of $\phi$, there are a disc $D$ with a structure of a simplicial complex, and a simplicial map $\psi\colon D \to X$ whose restriction to $\partial D=L$ is~$\gamma$.

\begin{figure}[h!]
	\begin{center}
		\includegraphics[scale=0.72]{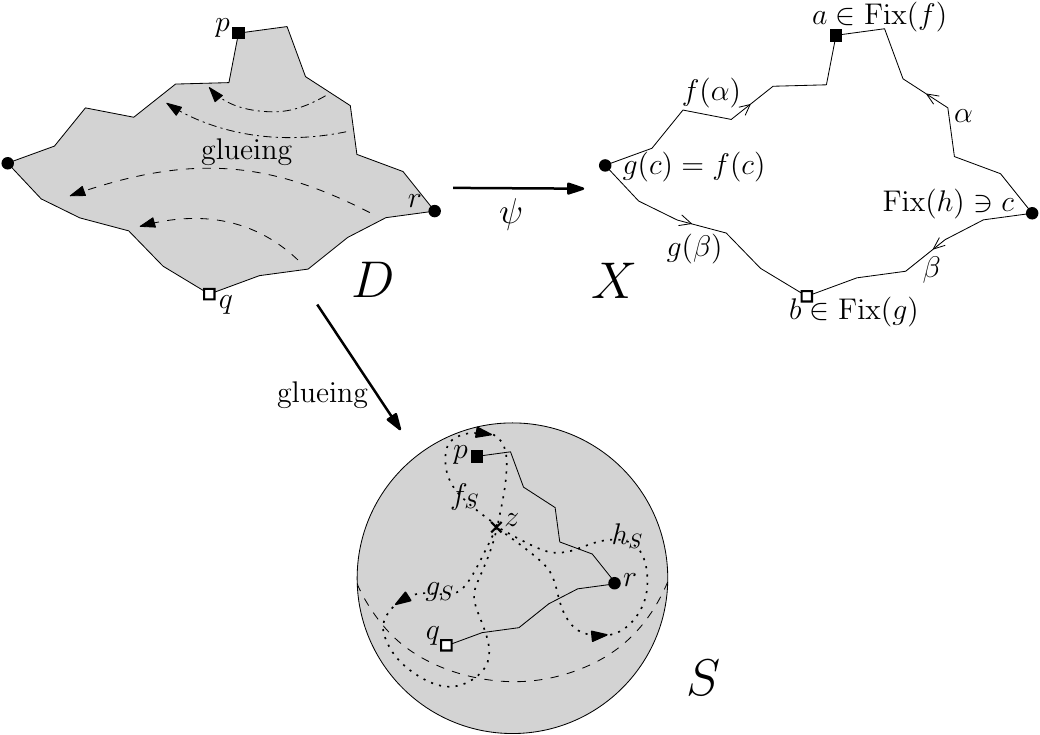}
	\end{center}
	\caption{Proof of Lemma~\ref{lem:min surface}, \textbf{Step 1.}}\label{fig:sphere}
\end{figure}

We label the points of $\partial D$ mapping under $\alpha^{-1}\beta$ to $a,b,c,$ by $p,q,r,$ respectively. We choose a basepoint $z\in D-\partial D$. Glueing in $D$ the parts of $\partial D$ that are the domains of $\alpha$ and $f(\alpha)$, and the parts of $\partial D$ that are the domains of $\beta$ and $g(\beta)$, we obtain a two-sphere. We identify this two-sphere with our template sphere $S$ from Definition~\ref{def:equivariant}, as in the bottom of Figure~\ref{fig:sphere}. This equips $S$ (and hence $\widetilde S$) with a structure of a $\Delta$-complex. We lift $D\subset \widetilde{S}$ so that $D$ contains the basepoint $\widetilde z$ of~$\widetilde{S}$. Then $D$ is a fundamental domain for the action of $\langle f_S, g_S\rangle$ on~$\widetilde{S}$ such that
\begin{enumerate}[(1)]
\item
$D$ and $f_SD$ (respectively, $D$ and $f_S^{-1}D$) share the domain of~$f(\alpha)$ (respectively,~$\alpha$), and
\item
$D$ and $g_SD$ (respectively, $D$ and $g_S^{-1}D$) share the domain of~$g(\beta)$ (respectively,~$\beta$).
\end{enumerate}

Let $\phi_*\colon \pi_1(S-\{p,q,r\},z) \to \langle f, g\rangle$ be the homomorphism mapping $f_S$ to $f$ and $g_S$ to $g$ as in Definition~\ref{def:equivariant}. We extend $\psi\colon D\to X$ to $\phi\colon \widetilde{S}\to X$ by defining, for each $w\in \langle f_S, g_S\rangle$, the restriction of $\phi$ to $wD\subset \widetilde{S}$ to be $\phi_*(w)\circ \psi \circ w^{-1}$. By~(1) and~(2), the map~$\phi$ is well defined on the intersections of $D$ with $f_SD$ and $g_SD$, and consequently $\phi$ is well defined on the intersections of $wD$ with $wf_SD$ and $wg_SD$ for each $w\in \langle f_S, g_S\rangle$.
Note that $\phi$ is equivariant, i.e.\ it satisfies for each $\tilde x \in \widetilde S$ and $w\in \langle f_S, g_S\rangle$ the formula $\phi(w \tilde x)=\phi_*(w)\phi(\tilde x)$. Indeed, for $\tilde x=w'x$ with $x\in D$ and $w'\in \langle f_S, g_S\rangle$, by definition we have $\phi(w \tilde x)=\phi(ww'x)=\phi_*(ww')\psi(x)$, while $\phi_*(w)\phi(\tilde x)=\phi_*(w)\phi(w'x)=\phi_*(w)\phi_*(w')\psi(x)$.
This completes the definition of an equivariant triangulation of $\widetilde S$.

\smallskip
\noindent \textbf{Step 2.} \emph{A minimal area equivariant triangulation has nondegenerate $\phi$.}
\smallskip

Consider an equivariant triangulation of $\widetilde S$ of minimal area.
Suppose that there is an edge $\widetilde e\subset \widetilde S$ with $\phi(\widetilde e)$ a vertex.
We will find an equivariant triangulation with strictly smaller area, which contradicts minimality.
Let~$e$ be the projection of $\widetilde e$ to~$S$.
Let $u,v$ be the endpoints of $e$.

\textbf{Case $u=v$.} In that case, assume for the moment that $u$ is distinct from $p,q,r$. Since $u=v$, we have that $e$ is an embedded closed path.
Let $B\subset S$ be the open disc bounded by $e$ containing at most one of $p,q,r$. Let $T$ be the triangle of $S$ adjacent to $e$ outside~$B$, and let $t,t'$ be the edges of $T$ distinct from $e$ (see Figure~\ref{fig:TB}, left). Note that $t\neq t'$ since otherwise $T$ would form the entire outside of $B$, which would contradict the assumption that there are at least two of $p,q,r$ outside $B$.

\begin{figure}[h!]
	\begin{center}
		\includegraphics[scale=0.66]{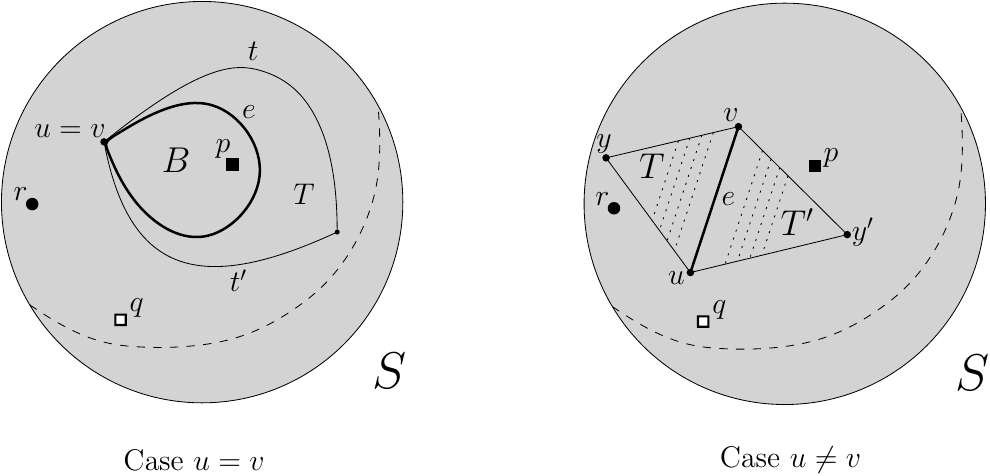}
	\end{center}
	\caption{Proof of Lemma~\ref{lem:min surface}, \textbf{Step 2.}}
\label{fig:TB}
\end{figure}

If $B$ does not contain a marked point $p,q$, or $r$, then we can remove $T\cup B$ from~$S$ and glue along $t$ and $t'$ (independent of whether all the vertices of $T$ coincide or not). This does not change the homeomorphism type of $S$ and decreases its area.
Possibly, we need to homotope the basepoint $z$ and the paths representing $f_S,g_S,h_S$ out of $T\cup B$ to keep track of the marking after this operation. If $B$ contains one of the marked points, say $p$, then we perform the same operation and additionally relabel $u$ by $p$. We can identify the modified $\widetilde{S}$ with a quotient of a subcomplex of the original~$\widetilde{S}$, and so we modify $\phi$ to be just the quotient of the restriction of the original $\phi$.

Going back to the possibility that $u$ is one of $p,q,r,$ say, $q$, we still perform the same operation, except that now $B$ will not contain $p$ (or $r$), since otherwise $\mathrm{Fix}(f)$ would intersect $\mathrm{Fix}(g)$.
Consequently we still have $t\neq t'$ and we can proceed as before.

\textbf{Case $u\neq v$.} Let $T,T'$ be the triangles of $S$ adjacent to $e$, and let $y,y'$ be the vertices opposite to $e$ in $T,T'$. If $T=T'$, then any lift $\widetilde T$ of $T$ to $\widetilde S$ is collapsed to a single vertex of $X$. Let $\widetilde {e}'$ be the edge of $\widetilde T$ that does not project to $e$. Then the projection $e'$ of $\widetilde {e}'$ has coinciding endpoints, which brings us back to the case $u=v$.

Thus we can assume that $T\neq T'$. We want to remove $T\cup T'$ from the triangulation and glue the resulting square in the boundary so that $u$ is identified with $v$. This amounts to collapsing segments of a foliation in $T\cup T'$, with leaves parallel to $e$ (see Figure~\ref{fig:TB}, right). This does not change the homeomorphism type of~$S$, as long as the leaves do not combine to circles.

Indeed, a single leaf (or a pair of leaves in a common triangle) cannot close up to a circle since the edges $yu$ and $yv$ are distinct, because $S$ does not contain a M\"obius band, and analogously the edges $uy',vy'$ are distinct. Furthermore, a pair of leaves in distinct triangles does not form a circle since otherwise $S$ would have only two triangles and consequently both $u,v$ would be in $\{p,q,r\}$, forcing some $\mathrm{Fix}(f),\mathrm{Fix}(g),\mathrm{Fix}(h)$ to intersect.

Removing $T\cup T'$ decreases the area of $S$. Note that as a result of this operation, vertices $p,q,r$ cannot become identified, since this again would mean that some $\mathrm{Fix}(f),\mathrm{Fix}(g),\mathrm{Fix}(h)$ intersect.

\smallskip
\noindent \textbf{Step 3.} \emph{A minimal area equivariant triangulation has $\phi$ a near immersion.}
\smallskip

By Step~2, $\phi$ is nondegenerate. Suppose that there is an edge $\widetilde e\subset \widetilde S$ with midpoint $\widetilde m$ where $\phi$ is not a local embedding. Again, we will reach a contradiction by showing that
the area can be decreased. Let $e$ and $m$ be the projections of $\widetilde e$ and $\widetilde m$ to $S$. Let $T,T'\subset S$ be the triangles containing $e$ and let $y,y'$ be the vertices opposite to~$e$ in $T,T'$.

To start with, note that $T\neq T'$. Indeed, if $T=T'$, then let $\tau$ be the line segment in~$T$ starting and ending at the distinct copies of $m$ in $\partial T$. Let $\widetilde\tau, w\widetilde{\tau}$ for $w\in \langle f_S, g_S\rangle$ be the two lifts of $\tau$ to $\widetilde S$ at $\widetilde m$, and let $\widetilde T$ be the lift of $T$ containing $\widetilde{\tau}$. Then $\phi(\widetilde T)$ is stabilised by $\phi_*(w)$, but not fixed pointwise, which is a contradiction.

\textbf{Case $y\neq y'$.} In that case removing $T\cup T'$ from the triangulation and glueing the resulting square in the boundary so that $y$ is identified with $y'$ is equivalent to the following. We collapse intervals of a foliation in $T\cup T'$ with leaves parallel to the union $\delta$ of line segments $ym\subset T,my'\subset T'$ (see Figure~\ref{fig:TB2}, left). Similarly as in Step~2 these leaves do not form circles, and thus collapsing them does not change the homeomorphism type of~$S$, while it decreases its area. Again as a result of this operation, vertices $p,q,r$ cannot become identified, since this would mean that some $\mathrm{Fix}(f),\mathrm{Fix}(g),\mathrm{Fix}(h)$ intersect.

\begin{figure}[h!]
	\begin{center}
		\includegraphics[scale=0.66]{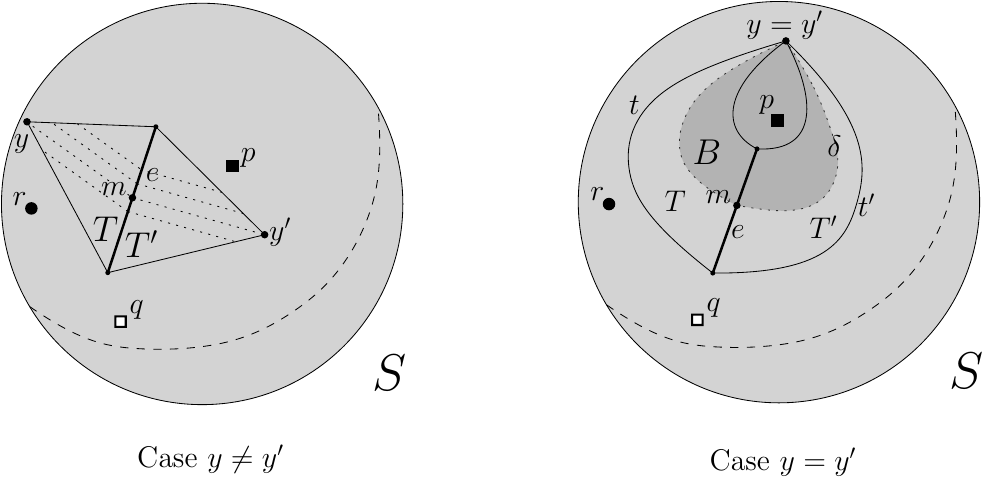}
	\end{center}
	\caption{Proof of Lemma~\ref{lem:min surface}, \textbf{Step 3.}}
	\label{fig:TB2}
\end{figure}

\textbf{Case $y=y'$.} In that case, assume for the moment that $y$ is distinct from $p,q,r$. Then $\delta$, defined as in the case $y\neq y'$, is an embedded closed path. Let $B\subset S$ be the open disc bounded by $\delta$ containing at most one of $p,q,r$. Let $t,t'$ be the edges of $T,T'$ outside $B$ (see Figure~\ref{fig:TB2}, right).
Note that $t\neq t'$ since otherwise halves of $T$ and $T'$ would form the entire outside of $B$, which would contradict the assumption that there are at least two of $p,q,r$ outside $B$. We can thus remove $T\cup T'\cup B$ and glue along $t,t'$ to decrease the area of $S$. If $B$ contained one of $p,q,r$, say, $p$, then we relabel $y$ by~$p$. If $y=y'$ is one of $p,q,r,$ say, $q$, then we define $B$ in the same way and we note that $B$ does not contain $p$ (or $r$) since otherwise $\mathrm{Fix}(f)$ would intersect $\mathrm{Fix}(g)$. Consequently we still have $t\neq t'$ and we can proceed as before.
\end{proof}

\section{Constructing axes}
\label{sec:loxodromic}

To prove Proposition~\ref{prop:loxodromic} we will need the following famous theorem of Masur establishing the existence of periodic trajectories in rational billiards.

\begin{defin}
A \emph{translation surface} $S$ is a surface obtained from identifying sides of finitely many polygons in $\R^2$ by translations.
This equips $S$ with a Riemannian metric of Gauss curvature $0$ outside a finite set $\Sigma$.
\end{defin}

\begin{lem}
\label{rem:hol}
Assume that a sphere $S$ has a piecewise smooth Riemannian metric that is smooth of Gauss curvature $0$ outside a finite set $\Sigma$.
Suppose that for each $v\in \Sigma$ the length of $\lk_{v}$ is commensurable with $\pi$. Then any finite branched cover of $S$ over $\Sigma$ has a further finite branched cover that is a translation surface.
\end{lem}
\begin{proof}

Since $\pi_1(S-\Sigma)$ is generated by the peripheral curves, the image of the holonomy map $\pi_1(S-\Sigma)\to O(2)$ is finite. Consequently, its kernel $K$ corresponds to a finite branched cover of $S$ over~$\Sigma$ that has trivial holonomy and is thus a translation surface. Thus any finite branched cover of $S$ over $\Sigma$ corresponding to a subgroup $F$ of $\pi_1(S-\Sigma)$ has a further finite branched cover corresponding to $K\cap F$ that is a translation surface.
\end{proof}

\begin{thm}[{\cite[Thm~2]{M}}]
\label{thm:masur}
Let $S$ be a translation surface. Then there is a closed local geodesic in $S-\Sigma$.
\end{thm}

We will also need the following.

\begin{lem}[{compare \cite[Lem 7.3]{BB}}]
\label{lem:BB2}
Let $S$ be a compact 
{\kol $\Delta$-complex with piecewise smooth Riemannian metric}
that is locally $\mathrm{CAT}(0)$ and in which each edge belongs to at least two triangles. Assume
that there is a vertex $v$ and points $\xi,\eta$ in~$\lk_v$ with $d_v(\xi,\eta)=\pi$. Then for any $\eps>0$ there
is a closed path $\beta_1\beta_2\beta_3$ in $S$ such that
\begin{itemize}
\item the paths $\beta_i$ are local geodesics,
\item the angles at the breakpoints between $\beta_2$ and $\beta_1,\beta_3$ are $>\pi-\eps$,
\item $\beta_i$ {\kol are transverse to edges and} do not pass through vertices except that $\beta_1$ starts at $v$ and $\beta_3$ ends at $v$, and
\item the starting direction $\xi'$ of $\beta_1$ and ending direction $\eta'$ of $\beta_3$ satisfy $d_v(\xi,\xi')<\frac{\eps}{2}, d_v(\eta,\eta')<\frac{\eps}{2}$.
\end{itemize}
\end{lem}

\begin{proof} We refer to the proof of \cite[Lem 7.3]{BB}, where the authors work in the universal cover of $S$ (which they call $X$). Once they construct their geodesic $\sigma$, define $\omega_2$ as the subpath of $\sigma$ between $P$ and $\varphi (P)$, and $\omega_1, \omega_3$ as the geodesics joining the endpoints of $\omega_2$ to $v,\varphi v$ in $P, \varphi(P)$. The projection of $\omega_1\omega_2\omega_3$ to $S$ is the required path $\beta_1\beta_2\beta_3$.
\end{proof}

{\kol \begin{lem}
\label{lem:new}
Let $X$ be a $\mathrm{CAT}(0)$ triangle complex and let $G$ be a group acting on~$X$. 
Suppose that there is a point $x$
\begin{enumerate}[(1)]
\item
in the interior of a triangle $T$ with negative Gauss curvature, or
\item in the interior of an edge $e$ with negative sum of
    geodesic curvatures in a pair of incident triangles
    $T,T'$.
\end{enumerate}
Then there is a $\mathrm{CAT}(0)$ triangle complex $\overline
X$ with an action of $G$, and a $G$-equivariant bilipschitz map
$X\to \overline X$ that is a composition of a subdivision and a
replacement of the piecewise smooth Riemannian metric, with a
vertex $u\in \overline X$ in the image of the interior of
$T,T',$ or $e,$ whose {\koll link} $\overline{\lk}_u$ in
$\overline X$ is either
\begin{itemize}
\item
a circle of length $>2\pi$, or
\item
a graph obtained from a family of disjoint circles of length $2\pi$ by glueing them along an arc $b$ of length $<\pi$.
Furthermore, $b$ corresponds to a triangle of~$X$ distinct from $T,T'$.
\end{itemize}
\end{lem}

\begin{proof} After possibly passing to the barycentric subdivision of $X$, we can assume that $G$ acts without inversions and no $g\in G$ sends an edge to a distinct edge in a common triangle.

In case~(1), choose a geodesic triangle $t$ in the interior of $T$ containing~$x$. Let $\overline t\subset \R^2$ be its comparison triangle, i.e.\ a geodesic triangle with the same edge lengths as~$t$. Let $\overline X$ be obtained from $X$ by a $G$-equivariant subdivision in which $t$ becomes a cell, and the $G$-equivariant replacement of the metric on $gt$ by that of $g\overline t$, for each $g\in G$. By \cite[II.2.9]{BH}, the angles of~$\overline t$ are larger than the corresponding angles of $t$. Consequently, $\overline X$ is $\mathrm{CAT}(0)$ and for each vertex $u$ of $\overline t$ we have that $\overline{\lk}_u$ is a circle of length $>2\pi$.

In case~(2), assume without loss of generality that the geodesic curvature $\kappa_-$ at~$x$ in $T$ is negative. Assume first that there is no triangle containing $e$ with positive geodesic curvature at $x$. Let $\kappa_-<\kappa<0$. Let $a$ be an arc in $e$ containing $x$ and with each point of $a$ of geodesic curvature
\begin{enumerate}[(i)]
\item
$\leq \kappa$ in $T$, and
\item
$\leq-\kappa$ in any other incident triangle.
\end{enumerate}
Let $u$ be a point in the interior of $T$ such that the
geodesic $ux$ in $T$ intersects $\partial T$ only at $x$, and
at angle $\frac{\pi}{2}$. After possibly shrinking~$a$,
denoting by $y,z$ its endpoints, we have that the geodesics
$uy,uz$ in $T$ intersect~$\partial T$ only at $y,z$, and at
positive angles. Let $t\subset T$ be the region bounded by $a,
uy,$ and $uz$. Let $\overline t\subset \R^2$ be its `comparison
region', that is, a region bounded by a circle arc $\overline
a$ of curvature $\kappa$ and length $|a|$ (the arc length of
$a$) and geodesics of lengths $d(u,y),d(u,z)$. See
Figure~\ref{fig:comptr}. (Note that $\overline t$ exists for
$|a|$ sufficiently small with respect to $d(u,x)$ and
$\frac{1}{|\kappa|}$.) Let $\overline X$ be obtained from $X$
by a $G$-equivariant subdivision in which $t$ becomes a cell,
and the $G$-equivariant replacement of the metric on $gt$ by
that of $g\overline t$, for each $g\in G$.

\begin{figure}[h!]
	\begin{center}
		\includegraphics[scale=0.65]{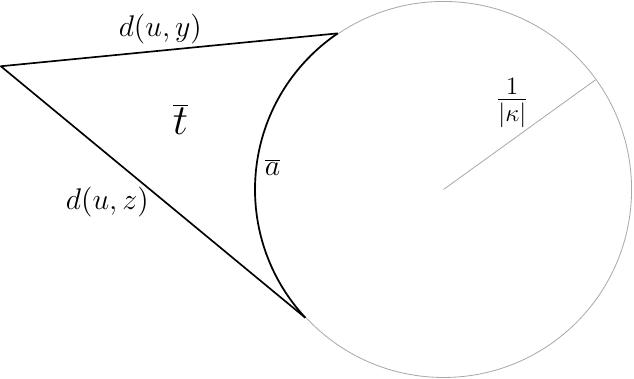}
	\end{center}
	\caption{`Comparison region' $\overline t$ in case (2).}
	\label{fig:comptr}
\end{figure}

We claim that the angles of $\overline t$ are larger than the corresponding angles of $t$. Indeed, let $s\subset \R^2$ be the region bounded by $\overline a$ and the geodesic joining the endpoints of $\overline a$. Note that $\overline t\cup s$ is a Euclidean triangle and $t\cup s$ (obtained by identifying $a$ with $\overline a$) is $\mathrm{CAT}(0)$ by condition~(i). By \cite[II.2.9]{BH}, the angles of $\overline t\cup s$ are larger than the corresponding angles of $t\cup s$, justifying the claim. By the claim and condition~(ii), we have that $\overline X$ satisfies conditions (a)--(c) of Section~\ref{sec:spaces} and $\overline{\lk}_u$ is a circle of length $>2\pi$.

However, since $\overline X$ might be locally infinite at $e$, we still need to justify that conditions (a)--(c) of Section~\ref{sec:spaces} imply that $\overline X$ is locally $\mathrm{CAT}(0)$ at $e$. Let $\mathrm{St}(e)$ be the union of all the closed triangles of $X$ containing $e$. Let $Y\subset \mathrm{St}(e)$ be the union of the triangles $T_+$ for which there exists a point on $e$ with positive geodesic curvature in $T_+$. By condition~(c) in $X$, there is at most one such triangle of given isometry type $T_0$ and given embedding $e\subset T_0$, so $Y$ has finitely many triangles. For each triangle $T_*$ of $\mathrm{St}(e)$ outside $Y$, denote $Y_{T_*}=Y\cup T_*$. By conditions (a)--(c) in $\overline X$, we have that the image $\overline Y_{T_*}$ of each $Y_{T_*}$ in $\overline X$ is $\mathrm{CAT}(0)$. Furthermore, for $\overline Y$ the image of~$Y$ in $\overline X$, the inclusion $\overline Y\subset \overline Y_{T_*}$ is an isometric embedding, by the claim and since points of $e$ have nonpositive geodesic curvature in all the triangles of $\overline X$ contained in the image of $T_*$. By \cite[II.11.3]{BH}, the union~$\overline{\mathrm{St}}(e)$ of~$\overline Y_{T_*}$ is $\mathrm{CAT}(0)$, as desired.

Finally, assume in case~(2) that there is a triangle $T_+$
containing $e$ with positive geodesic curvature $\kappa_+$ at
$x$. If $T'=T_+$, then we have $\kappa_-<-\kappa_+$.
Consequently, we can repeat the argument above assuming
$\kappa_-<\kappa<-\kappa_+$ instead of $\kappa_-<\kappa<0$.

If $T'\neq T_+$, then we choose an arc $a$ in $e$ containing
$x$, with endpoints $y,z,$ such that the geodesic $yz$ in $T_+$
intersects $\partial T_+$ only at $y,z$, and at positive
angles. Let $s\subset T_+$ be the region bounded by $a$ and
$yz$. To form $\overline X$, we replace the metric in~$T_+$ by
that of $T_+\setminus s$, and the metric in each triangle
$T_*\neq T_+$ containing $e$ by that of $T_*\cup s$ (which is
smooth with respect to a subdivision including $T_*\cap s$). We
perform the same replacement on the $G$-orbit of $T_+$ and each
$T_*$. Note that $\overline X$ satisfies conditions (a)--(c) of
Section~\ref{sec:spaces}, and hence is $\mathrm{CAT}(0)$ as
before. Moreover, both $\overline{\lk}_y,\overline{\lk}_z$ have
the form described in the second bullet, with $b$ corresponding
to $T_+\neq T,T'$. Note that in both
$\overline{\lk}_y,\overline{\lk}_z$, the arc $b$ has length
$<\pi$ since the geodesic $yz$ in $T_+$ met $\partial T_+$ at
positive angles.
\end{proof}}

{\kol In the proof of Proposition~\ref{prop:loxodromic} we will use the following notions.} If $X$ is a $\mathrm{CAT(0)}$ space, an element $g\in G$ is \emph{loxodromic} if there is a geodesic $\widetilde \omega\subset X$ (called an \emph{axis})
such that $g$ preserves $\widetilde \omega$ and acts on it as a nontrivial translation. A loxodromic element does not have a fixed point in $X$.

\begin{proof}[Proof of Proposition~\ref{prop:loxodromic}.]
{\kol We equip $\widetilde S$ and $S$ with the piecewise smooth Riemannian metric pulled back from $X$ via $\phi$.
Since $\phi$ is a near-immersion, by conditions (a)--(c) of Section~\ref{sec:spaces}, we have that $S-\{p,q,r\}$ is locally CAT(0). After replacing~$S$ with a surface that is a finite branched cover of $S$ over $\{p,q,r\}$, we can assume that $S$ (which is no longer a sphere) is locally CAT(0) at every point. Let $\Sigma$ be the vertex set of $S$.

Consider first the case, where there is no point in the interior of a triangle of $S$ with negative Gauss curvature, and no point in the interior of an edge of $S$ with negative sum of its two geodesic curvatures.
Then} $S-\Sigma$ is smooth with Gauss curvature~$0$. Suppose first that for each $v\in \Sigma$, the length of $\lk_{v}$ is commensurable with $\pi$. By Lemma~\ref{rem:hol}, there is a finite branched cover $S'$ of $S$ over $\Sigma$
that is a translation surface. By Theorem~\ref{thm:masur}, $S'$ has a closed local geodesic $\omega'$
outside the vertex set.
{\kol Let $\omega$ be the projection of $\omega'$ to $S$. Let $\widetilde{\omega}$ be a lift of $\omega$ to $\widetilde{S}$, which is an axis for some $w\in \langle f_S, g_S\rangle$. For $e$ an edge of $\widetilde S$  intersected by $\widetilde{\omega}$, let $T,T'$ be the triangles of $\widetilde S$ containing $e$. Since $\phi$ is a near-immersion, we have $\phi(T)\neq \phi(T')$. Furthermore, since the sum of the geodesic curvatures at any point~$x$ of~$e$ in~$\phi(T)$ and~$\phi(T')$ equals $0$, the geodesic curvature at $x$ in any triangle of $X$ distinct from $\phi(T),\phi(T')$ is nonpositive. Consequently, $\phi(T\cup T')$ is locally convex at $\phi(e)$ in $X$. Thus $\phi(\widetilde{\omega})$ is a local, hence global, geodesic in~$X$. This implies that $\phi_*(w)\in \langle f,g\rangle$ is loxodromic.}

{\kol Suppose} now that for some $v\in\Sigma$ the length of $\lk_{v}$ is not commensurable with~$\pi$.
Let $\widetilde v$ be a lift of $v$ to $\widetilde S$. If $\lk_{\widetilde{v}}$ is a circle, consider the closed immersed edge-path $\gamma\colon\lk_{\widetilde{v}}\to\lk_{\phi(\widetilde{v})}$ induced by $\phi$. This path does not satisfy the conclusion of Corollary~\ref{thm:graph}. Thus there are points $\widetilde{\xi},\widetilde{\eta}\in \lk_{\widetilde{v}}$ such that their images in $\lk_{\phi(\widetilde{v})}$ are at distance $>\pi+\delta$ for some $\delta>0$.

If $\lk_{\widetilde{v}}$ is a line, we construct $\widetilde{\xi},\widetilde{\eta}$ in the following way.
Assume without loss of generality that $\widetilde{v}$ is fixed by $f_S$. First, we claim that $X$ does not have rational angles with respect to the action of {\kol $\langle f,g\rangle$}. Otherwise, let $\Lambda\subset \lk_{\phi(\widetilde{v})}$ be the discrete set from Definition~\ref{def:rational}. Then the complementary components in $\lk_{\widetilde{v}}$ of $\phi^{-1}(\Lambda)$ are finite, and of length commensurable with $\pi$. Moreover, since $f$ preserves $\Lambda$, we have that $f_S$ preserves $\phi^{-1}(\Lambda)$. This contradicts the assumption that the length of $\lk_{v}$ is not commensurable with $\pi$, justifying the claim.

According to our hypotheses this means that either $X$ is locally finite, or each element of $G$ fixing a point of $X$ has finite order.
In both cases there is a directed edge $e$ in $\lk_{\widetilde{v}}$ and $k\geq 1$ such that $\phi(e)=\phi(f_S^ke)$. Thus the path in $\lk_{\widetilde{v}}$ from the endpoint of $e$ to the endpoint of $f^k_Se$ maps under $\phi$ to a closed edge-path that does not satisfy the conclusion of Corollary~\ref{thm:graph}, and we obtain $\widetilde{\xi},\widetilde{\eta}$ as before.

Let $\xi,\eta$ be the projections of $\widetilde{\xi},\widetilde{\eta}$ to $\lk_{v}$. Orbits of the rotation by $\pi$ on the circle $\lk_{v}$ of length not commensurable with $\pi$ are dense. Thus in $\lk_{v}$ we can find $\xi_1=\xi, \xi_2, \ldots, \xi_{2n}$ with $d_{v}(\xi_i,\xi_{i+1})=\pi$ and $d_{v}(\xi_{2n},\eta)<\frac{\delta}{2}$. Inspired by \cite[Lem 7.4]{BB}, we construct the following path $\widetilde{\omega}=\omega_1\cdots \omega_{6n}$ in $\widetilde{S}$.
To start with, we put $\eps=\frac{\delta}{12n}$ and apply Lemma~\ref{lem:BB2} to $\xi_1,\xi_2$. We define $\omega_1\omega_2\omega_3$ to be the lift of $\beta_1\beta_2\beta_3$ starting in the direction at distance $<\frac{\eps}{2}$ from $\widetilde{\xi}$. Let $\widetilde v_2$ be the endpoint of $\omega_3$ and let $\widetilde{\xi}_2$ in $\lk_{\widetilde {v}_2}$ be the lift of $\xi_2$ at distance $<\frac{\eps}{2}$ from the ending direction of $\omega_3$. Since $d_{\kol v}(\xi_2,\xi_{3})=\pi$, there is a lift $\widetilde{\xi}_3$ of $\xi_3$ in $\lk_{\widetilde v_2}$ with $d_{\widetilde v_2}(\widetilde{\xi}_2,\widetilde{\xi}_{3})=\pi$. Apply now Lemma~\ref{lem:BB2} to $\xi_3,\xi_4$ and define $\omega_4\omega_5\omega_6$ to be the lift of the resulting $\beta_1\beta_2\beta_3$ starting in the direction at distance $<\frac{\eps}{2}$ from $\widetilde{\xi}_3$ etc. See Figure~\ref{fig:omega6}. The endpoint of $\omega_{6n}$ has the form $w\widetilde {v}$ for some $w\in \langle f_S,g_S\rangle$, and since $d_{v}(\xi_{2n},\eta)\leq\frac{\delta}{2}$, we can choose $w$ so that the ending direction of $\omega_{6n}$ is at distance $<\frac{\eps}{2}+\frac{\delta}{2}$ from $w\widetilde \eta$.

\begin{figure}[h!]
	\begin{center}
	\includegraphics[scale=0.65]{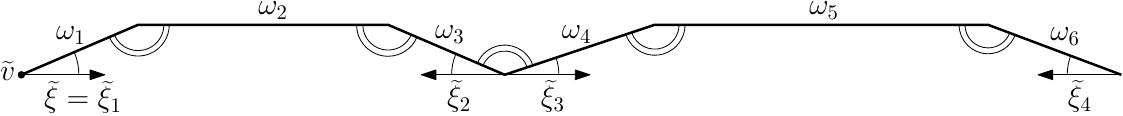}
	\end{center}
	\caption{Path $\omega_1\cdots \omega_6.$ Angles indicated with a single arc are $<\frac{\eps}{2}$. Angles indicated with a double arc are $> \pi-\eps$.}
	\label{fig:omega6}
\end{figure}

{\kol Note that since $\phi$ is a near-immersion and $\omega_i$ are transverse to the edges of $\widetilde S$, we have that $\phi(\omega_i)$ are geodesics in $X$.} Let $x,y=\phi_*(w)x\in X$ be the endpoints of $\phi(\widetilde{\omega})$, and denote by $\xi'$ and $\eta'$ the starting and ending directions of $\phi(\widetilde{\omega})$. Then $d_x(\phi(\widetilde{\xi}),\xi')< \frac{\eps}{2}$ and $d_y(\phi(w\widetilde{\eta}), \eta')< \frac{\eps}{2}+\frac{\delta}{2}$. Let $\alpha$ be the geodesic from $x$ to $y$ in~$X$, and denote by $\xi''$ and $\eta''$ the starting and ending directions of $\alpha$. Note that the angles at all the breakpoints of $\phi(\widetilde{\omega})$ are $> \pi-\eps$.
By \cite[Lem~2.5]{BB} we have $d_x(\xi',\xi'')+d_y(\eta',\eta'')< (6n-1)\eps$. We thus have:
\begin{align*}
d_y(\phi_*(w)\xi'',\eta'')&\geq  d_y(\phi_*(w)\xi',\eta')-d_y(\phi_*(w)\xi'',\phi_*(w)\xi')-d_y(\eta',\eta'')\\
&>d_y(\phi_*(w)\xi',\eta')-(6n-1)\eps\\
&\geq  d_y\big(\phi(w\widetilde {\xi}),\phi(w\widetilde {\eta})\big)-d_y\big(\phi_*(w)\xi',\phi(w\widetilde {\xi})\big) -d_y\big(\phi(w\widetilde {\eta}),\eta'\big)-(6n-1)\eps\\
&>
(\pi+\delta)-\frac{\eps}{2}-\Big(\frac{\eps}{2}+\frac{\delta}{2}\Big)-(6n-1)\eps=\pi.
\end{align*}
Consequently the concatenation of {\kol $\phi_*(w)^k\alpha$}, for $k\in \Z$, is a local (hence global) geodesic, and thus $\phi_*(w)$ is loxodromic.

{\kol It remains to consider the case, where there is a point
in the interior of a triangle~$T$ of $\widetilde S$ with
negative Gauss curvature, or a point in the interior of an edge
of~$\widetilde S$ with negative sum of its geodesic curvatures
in incident triangles $T,T'$. Let $\overline X$ be the complex
obtained from $X$ through Lemma~\ref{lem:new} applied to
$\phi(T)$ and~$\phi(T'),$ for $G=\langle f,g\rangle,$ with a
vertex $u$ satisfying one of the two bullets of
Lemma~\ref{lem:new}. Let $\widetilde v\in T\cup T'$ be the
preimage of $u$ under the composition $\overline {\phi}\colon
\widetilde S\to X\to\overline X$, and let $v$ be the projection
of~$\widetilde v$ to~$S$. Replace the $\Delta$-complex
structure and the piecewise smooth Riemannian metric on
$\widetilde S$ by the one pulled back from $\overline {X}$ via
$\overline{\phi}$.

Suppose first that $\overline{\lk}_u$ is a circle of length $\geq 2(\pi+\delta)$, which is then isometric to $\lk_{\widetilde{v}}$ and $\lk_v$. Let
$\xi_1,\xi_4$ be antipodal points in~$\lk_{v}$, and let $\xi_2,\xi_3\in \lk_{v}$ be also antipodal points with
$d_v(\xi_1,\xi_2)=d_v(\xi_3,\xi_4)=\pi$. We take $\eps=\frac{\delta}{5}$ and use Lemma~\ref{lem:BB2} as before to obtain appropriate paths
$\omega_1\omega_2 \omega_3$ and $\omega_4\omega_5 \omega_6$ in~$\widetilde S$.  Let $\alpha_1$ and $\alpha_2$ be the geodesics in $\overline X$ with
the same endpoints as $\phi(\omega_1\omega_2 \omega_3)$ and $\phi(\omega_4\omega_5 \omega_6)$. Then by \cite[Lem~2.5]{BB}, the starting direction of
$\alpha_2$ is at angle $< 2\eps$ from the starting direction of $\phi(\omega_4)$, which is at angle $>\pi+\delta-2\frac{\eps}{2}$ from the ending
direction of~$\phi(\omega_3)$, which is in turn at angle $< 2\eps$ from the ending direction of $\alpha_1$. Thus $\alpha_1\alpha_2$ is a geodesic.
Similarly, for appropriate $w\in \langle f_S,g_S\rangle$, we have that $\alpha_2\phi_*(w)\alpha_1$ is a geodesic. Then the concatenation of
$\phi_*(w)^k(\alpha_1\alpha_2),$ for~$k\in \Z$, is a geodesic, and $\phi_*(w)$ acts on it as a nontrivial translation. Consequently, $\phi_*(w)$ is
loxodromic w.r.t.\ the action on~$\overline X$ and in particular has no fixed point in $\overline X$. Thus $\phi_*(w)$ has no fixed point in $X$.

Finally, suppose that $\overline{\lk}_u$ is a graph obtained
from a family of disjoint circles $C_1,C_2,\ldots$ of length
$2\pi$ by glueing them along an arc $b$ of length $<\pi$. We
can assume that triangles $\phi(T)$ and $\phi(T')$ correspond
to $C_1\setminus b$ and $C_2\setminus b$. Then
$\overline{\phi}$ induces an embedding of $\lk_{\widetilde{v}}$
into $\overline{\lk}_u$ allowing us to identify
$\lk_{\widetilde{v}}$, and hence $\lk_v$, with the circle
$(C_1\cup C_2)\setminus b$. Under this identification, let
$\xi_1,\xi_3\in C_1\setminus b$ and $\xi_2,\xi_4\in
C_2\setminus b$ be points at distance $\frac{\pi}{2}$ from the
endpoints of~$b$, with $d_v(\xi_1,\xi_2)=d_v(\xi_3,\xi_4)=\pi$.
Choosing $\delta=|b|$, we can then use Lemma~\ref{lem:BB2} as
before to obtain an element $\phi_*(w)$ that is loxodromic
w.r.t.\ the action on~$\overline X$.}
\end{proof}

\section{Triples of generators}
\label{sec:3}

In this Section we complete the proof of Theorem~\ref{thm:torsion}.

\begin{prop}
	\label{prop:triples}
	Let $f,g,h$ be automorphisms of a $\mathrm{CAT}(0)$ triangle complex $X$ satisfying condition (i),(ii), or (iii) of Theorem~\ref{thm:torsion}, {\kol with respect to $\langle f,g,h\rangle$}. Suppose that all $\mathrm{Fix}(f)\cap \mathrm{Fix}(g), \mathrm{Fix}(f)\cap\mathrm{Fix}(h),$ and $\mathrm{Fix}(g)\cap\mathrm{Fix}(h)$ are nonempty. Then $\mathrm{Fix}(f)\cap\mathrm{Fix}(g)\cap\mathrm{Fix}(h)$ is nonempty or
	$\langle f,g,h\rangle$ contains an element {\kol with no fixed point in $X$}.
\end{prop}

In the proof we will need the following notion. Let $\Delta$ be a disc with decomposition of its boundary into three paths $\partial \Delta=P\cup Q\cup R$. An \emph{admissible triangulation} of $\Delta$ w.r.t.\ $(f,g,h)$ is a structure on $\Delta$ of a  $\Delta$-complex, with $P\cap Q, Q\cap R, P\cap R$ among the vertices, together with a simplicial map $\psi$ from $\Delta$ to $X$ that sends $P,Q,R$ into $\mathrm{Fix}(f),\mathrm{Fix}(g),\mathrm{Fix}(h),$ respectively.

\begin{lem}
\label{lem:min surface2}
Let $f,g,h$ be automorphisms of a $\mathrm{CAT}(0)$ triangle complex~$X$ 
acting without inversions. Suppose that $\mathrm{Fix}(f),\mathrm{Fix}(g),\mathrm{Fix}(h)$ pairwise intersect but their triple intersection is empty. Then there exists an admissible triangulation of $\Delta$ w.r.t.\ $(f,g,h)$ with $\psi$ a near-immersion.
\end{lem}

\begin{proof}
Since $f,g,h$ act without inversions, their fixed point sets are subcomplexes. Thus an admissible triangulation of $\Delta$ exists by Theorem~\ref{thm:Z}. Suppose now that $\Delta$ has minimal area among admissible triangulations w.r.t.\ $(f,g,h)$.

We first prove that $\psi$ is nondegenerate. Indeed, suppose that there is an edge $e\subset \Delta$ with $\psi(e)$ a vertex, and let $u,v$ be the endpoints of $e$. If $e$ lies in $\partial \Delta$, then let $T$ be the triangle containing $e$. We then remove $e$ and $T$ from $\Delta$ and we identify the two remaining sides of $T$. This decreases the area of $\Delta$, which contradicts minimality.

Suppose then that $e$ is contained in two triangles $T$ and $T'$. If $u=v$,
then let $B\subset \Delta$ be the open disc bounded by $e$. Let $T$ be the triangle of $\Delta$ adjacent to~$e$ outside~$B$, and let $t,t'$ be the edges of $T$ distinct from $e$. We remove $T\cup B$ from $\Delta$ and glue along $t$ and $t'$. This does not change the homeomorphism type of $\Delta$, and decreases its area, which is a contradiction.

If $u\neq v$ and $T=T'$, we land back in the case $u=v$. Otherwise, we wish to remove $T\cup T'$, or more precisely to collapse intervals of the foliation in $T\cup T'$ parallel to~$e$. Unless $e$ has both endpoints on $\partial \Delta$, this does not change the homeomorphism type of~$\Delta$ (see Figure~\ref{fig:L62}, left). If both $u$ and $v$ are on the path $P$ (resp.\ $Q,R$), then together with $T$ and~$T'$ we remove the entire disc bounded by $e$ and a subpath of $P$ (resp.\ $Q,R$), and we identify the two remaining edges of $T$ or $T'$ (see Figure~\ref{fig:L62}, right). If the endpoints of $e$ lie on two distinct paths, say $P$ and $Q$, then we do the same using a subpath of $P\cup Q$. The vertex $P\cap Q$ is replaced here by $u$ identified with $v$.

\begin{figure}[h!]
	\begin{center}
		\includegraphics[scale=0.6]{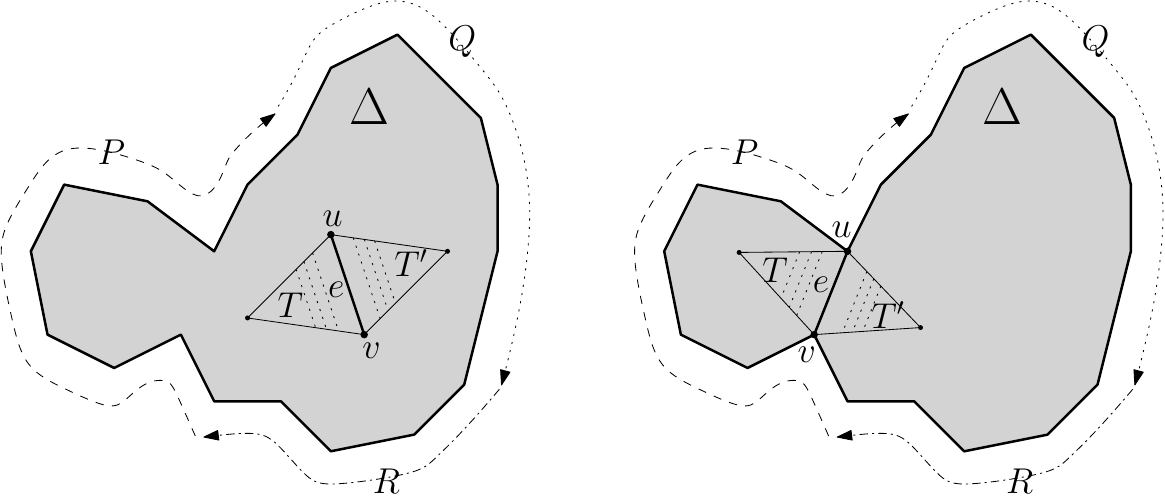}
	\end{center}
	\caption{Proof of Lemma~\ref{lem:min surface2}.}
	\label{fig:L62}
\end{figure}

Finally, we prove that $\psi$ is a near-immersion. Suppose that there is an edge $e\subset \Delta$ with midpoint $m$ where $\psi$ is not a local embedding. Let $T,T'\subset \Delta$ be the triangles containing $e$ and let $y,y'$ be the vertices opposite to $e$ in $T,T'$. We have $T\neq T'$, since otherwise $\psi$ collapses the edge of $T$ distinct from $e$ to a vertex. Let $\delta$ be the path that is the union of line segments $ym\subset T$ and $my'\subset T'$.

If $y=y'$, then let $B\subset \Delta$ be the open disc bounded by $\delta$ and let $t,t'$ be the edges of $T,T'$ outside $B$.
We can remove $T\cup T'\cup B$ and glue along $t,t'$ to decrease the area of $\Delta$, which is a contradiction.

If $y\neq y'$, then we collapse intervals of the appropriate foliation in $T\cup T'$. Unless $\delta$ has both endpoints on $\partial \Delta$, this does not change the homeomorphism type of~$\Delta$.
If both endpoints of $\delta$ are on the path $P$ (resp.\ $Q,R$), then together with $T$ and $T'$ we remove the entire disc bounded by $\delta$ and a subpath of~$P$ (resp.\ $Q,R$), and we identify the two remaining edges of $T,T'$. If the endpoints of $\delta$ lie on two distinct paths, say $P$ and $Q$, then we do the same using a subpath of $P\cup Q$ ($P\cap Q$ is replaced here by $y$ identified with $y'$).
\end{proof}

\begin{lem}
\label{lem:discsphere}
Let $f,g,h$ be automorphisms of a $\mathrm{CAT}(0)$ triangle complex~$X$. Suppose that we have an admissible triangulation of $\Delta$ w.r.t.\ $(f,g,h)$ with $\psi$ a near-immersion and satisfying the following property.
\begin{description}
\item[($\star$)]
 For any edge $e$ of $P$ (resp.\ $Q,R$), the triangle containing $e$ is not mapped by $\psi$ to $\mathrm{Fix}(f)$ (resp.\ $\mathrm{Fix}(g),\mathrm{Fix}(h)$).
\end{description}
Then for $f'=h\circ f, g'=h\circ g$ we have an equivariant triangulation of $\widetilde S$ w.r.t.\ $(f',g')$ with $\phi$ a near-immersion.
\end{lem}
\begin{proof}
We label the vertices of $\Delta$ by $p=P\cap R,q=Q\cap R, r=P\cap Q$. Let $D$ be obtained from $\Delta$ by attaching a second copy of $\Delta$ along the side $R$. Extend $\psi$ to a map from $D$ to $X$ defined as $h\circ \psi$ on the attached copy of $\Delta$.

Denote $\alpha=\psi|_P, \beta=\psi|_Q$ to be directed paths starting at $\psi(r)$. Then the restriction of $\psi$ to $\partial D$ is $\alpha^{-1}\beta g'(\beta^{-1})f'(\alpha)$. Glueing in~$D$ the parts of $\partial D$ that are the domains of $\alpha$ and $f'(\alpha)$, and the parts of $\partial D$ that are the domains of $\beta$ and $g'(\beta)$, we obtain a template two-sphere $S$ from Definition~\ref{def:equivariant} (as in the bottom of Figure~\ref{fig:sphere}).

Let $\phi_*\colon \pi_1(S-\{p,q,r\},z)\to \langle f', g'\rangle$ be the homomorphism mapping $f_S$ to $f'$ and $g_S$ to $g'$. We extend $\psi\colon D\to X$ to $\phi\colon \widetilde{S}\to X$ by defining, for each $w\in \langle f_S, g_S\rangle$, the restriction of $\phi$ to $wD\subset \widetilde{S}$ to be $\phi_*(w)\circ \psi \circ w^{-1}$. This forms an equivariant triangulation of $\widetilde S$ w.r.t.\ $(f',g')$, with the equivariant map $\phi \colon \widetilde S \to X$.

Since $\phi$ is equivariant, to prove that it is a near-immersion, it suffices to justify that it is a local embedding at the midpoint of any edge $e$ in the interior of $D$, or in $P$ or $Q$. If $e$ lies in the interior of $D$ but not in $R$, then this follows from the assumption that $\psi$ is a near-immersion. If $e$ lies in $R$ and is contained in a triangle~$T$ of $\Delta$, then the two triangles of $D$ containing $e$ are send by $\phi$ to $\psi(T)$ and $h\psi(T)$. By property {\bf ($\star$)}, we have $h\psi(T)\neq \psi(T)$ and so $\phi$ is a local embedding at the midpoint of $e$ as well. Finally, if $e$ lies in, say, $P$, and is contained in a triangle~$T$ of $\Delta$, then the two triangles of $\widetilde S$ containing $e$ are $T$ and the image under~$f^{-1}_S$ of the second copy of~$T$ in $D$. These two triangles are sent by $\phi$ to $\psi(T)$ and $(f')^{-1}h\psi(T)=f^{-1}\psi(T)$. By property {\bf ($\star$)}, we have $f^{-1}\psi(T)\neq \psi(T)$ and so $\phi$ is a local embedding at the midpoint of $e$ as before.
\end{proof}

\begin{proof}[Proof of Proposition~\ref{prop:triples}]
After possibly passing to the barycentric subdivision of $X$ (which obviously preserves the conditions~(i) and~(ii) of Theorem~\ref{thm:torsion}, and preserves condition~(iii) by Lemma~\ref{lem:barycentric_preserves_rational}), we can assume that $f,g,h$ act without inversions. Suppose that $\mathrm{Fix}(f)\cap\mathrm{Fix}(g)\cap\mathrm{Fix}(h)=\emptyset$. Then by Lemma~\ref{lem:min surface2}, there is an admissible triangulation of $\Delta$ with $\psi$ a near-immersion.
Moreover, by possibly passing to a subcomplex of $\Delta$, we can assume that $\Delta$ satisfies property ($\star$).

By Lemma~\ref{lem:discsphere},  for $f'=h\circ f, g'=h\circ g$ we have an equivariant triangulation of~$\widetilde S$ w.r.t.\ $(f',g')$ with $\phi$ a near-immersion. Thus by Proposition~\ref{prop:loxodromic}, $\langle f',g'\rangle$ contains an element {\kol with no fixed point in $X$}.
\end{proof}

\begin{proof}[Proof of Theorem~\ref{thm:torsion}]
Suppose by contradiction that for each $f\in G$ the set $\mathrm{Fix}(f)$ is nonempty. We will prove that for any finite set of elements $f_1,\ldots, f_n\in G$ the intersection $\mathrm{Fix}(f_1)\cap \cdots \cap \mathrm{Fix}(f_n)$ is nonempty. The case $n=2$ follows from Proposition~\ref{prop:pairs}. Consequently, the case $n=3$ follows from Proposition~\ref{prop:triples}.
Since the fixed point sets are convex, the cases $n\geq 4$ follow from Helly's Theorem \cite[Thm~1.1]{I}. Setting $f_1,\ldots, f_n$ to be the generators of $G$ we obtain that $G$ fixes a point of $X$, which is a contradiction.
\end{proof}

\section{Irrational loops have diameter $>\pi$}
\label{sec:graph}
In this section we prove Theorem~\ref{thm:graph+}.

The key ingredient in the proof is a reformulation of a theorem of Dehn~\cite{Dehn}, which states that a rectangle can be tiled by finitely many squares if and only if its sides are commensurable.

Let $(X,\mu_X)$, $(Y,\mu_Y)$ be measure spaces of finite measure. We will write $\mu$ instead of $\mu_X$ and $\mu_Y$ for brevity. We say that a pair $(A,B)$, such that $A \subseteq X$, $B \subseteq Y$ are measurable, is a \emph{rectangle in $X \times Y$} with \emph{side lengths $\mu(A)$ and $\mu(B)$}. A rectangle is a \emph{square} if its side lengths are equal.
We say that a  collection  $(A_1,B_1),(A_2,B_2), \ldots, (A_k,B_k)$ of rectangles is a \emph{rectangle tiling of $X \times Y$} if
\begin{itemize}
	\item $\bigcup_{i=1}^{k}A_i \times B_i = X \times Y$, and
	\item $\mu(A_i \cap A_j)\mu(B_i \cap B_j)=0$ for $1 \leq i < j \leq k$.
\end{itemize}
We say that a rectangle tiling is \emph{a square tiling} if it consists of squares.

\begin{thm}\label{thm:Dehn}
		Let $(A_1,B_1),$ $(A_2,B_2), \ldots, (A_k,B_k)$ be a square tiling of $X \times Y$. Then $\mu(X)$, $\mu(Y)$, $\mu(A_1),\ldots,\mu(A_k)$ are all commensurable.
\end{thm}

While it is not hard to see that Theorem~\ref{thm:Dehn} is equivalent to the theorem of Dehn stated above we include a proof both for completeness, and because we will need a technical modification of the result. The proof we give is essentially the one used by Hadwiger~\cite{Hadwiger} to prove a multidimensional generalisation of Dehn's theorem, except that we follow the presentation from Aigner and Ziegler~\cite[Chap 29]{AZ}, which avoids the reliance on the existence of Hamel basis of $\bb{R}$ over $\bb{Q}$ and thus the dependence on the Axiom of Choice.

\begin{proof}[Proof of Theorem~\ref{thm:Dehn}]
	We start by showing that it suffices to consider the case when $X$ and $Y$ are finite. The reduction is fairly straightforward.
	 Define an equivalence relation $\sim$ on $X$, by setting $x_1 \sim x_2$ if for every $1 \leq i \leq k$ either $\{x_1,x_2\} \in A_i$ or $\{x_1,x_2\} \cap A_i =  \emptyset$. Clearly there are finitely many equivalence classes of elements of $X$ with respect to this relation, and we can identify the elements of $X$ lying in the same equivalence class. Thus we can assume that $X$ and, symmetrically, $Y$ are finite. Further, we may assume without loss of generality that every element of $X$ and $Y$ has a positive measure. It follows that for every $x \in X$ and $y \in Y$ there exists a unique index $i$ such that $x \in A_i$ and $y \in B_i$.
	
	 Let $V$ be a vector space over $\bb{Q}$ spanned by $\{\mu(x)\}_{x \in X} \cup \{\mu(y)\}_{y \in Y}$.
  Suppose first that $\mu(X)$ and $\mu(Y)$ are not commensurable. Then there exists a linear function $f: V \to \bb{Q}$ such that $f(\mu(X))=1$ and $f(\mu(Y))=-1$. By linearity we have
   \begin{align}\label{e:Dehn}
   	 -1 &= f(\mu(X))f(\mu(Y))= \sum_{x \in X}f(\mu(x))\sum_{y \in Y}f(\mu(y)) \notag \\ &=\sum_{i=1}^k f(\mu(A_i))f(\mu(B_i)) =  \sum_{i=1}^k f^2(\mu(A_i)) \geq 0,
   \end{align}	
	 yielding the desired contradiction.	
	
	 Suppose finally that $\mu(A_j)$ is not commensurable with $\mu(X)$ for some $0 \leq j \leq k$. Define a linear function $g: V \to \bb{Q}$ such that $g(\mu(X))=0$ and $g(\mu(A_j))=1$. Substituting $g$ instead of $f$ in (\ref{e:Dehn}) yields
$$
	0 = g(\mu(X))g(\mu(Y))=\sum_{i=1}^k g^2(\mu(A_i)) \geq g^2(\mu(A_j)) =1,
$$
 a contradiction.
\end{proof}

We will also need a technical variant of Theorem~\ref{thm:Dehn}.

\begin{lem}\label{lem:Dehn+}
	Let $(A_1,B_1),$ $(A_2,B_2), \ldots, (A_k,B_k)$ be a rectangle  tiling of $X \times Y$, such that for some $q,r \in \bb{R}_+$ and $a \in \bb{Q}$ we have
	\begin{itemize}
		\item $\mu(X)=2q+ar$, $\mu(Y)=q + (a/2-1)r$,
		\item  $(A_1,B_1)$ and $(A_2,B_2)$ are rectangles with sides $r$ and $q+r$,
		\item $(A_3,B_3), \ldots, (A_k,B_k)$ are squares,
		\item $\mu(A_3),\ldots,\mu(A_j)$ are commensurable with $r$ for some $3 \leq j \leq k$, and $\sum_{i=3}^j\mu^2(A_j) > (a-4)r^2$
	\end{itemize}	
	Then $q$ and $r$ are commensurable.
\end{lem}

\begin{proof}
The proof parallels the proof of Theorem~\ref{thm:Dehn} above with minor modification to the calculation (\ref{e:Dehn}). We
define the vector space $V$ as in the proof of Theorem~\ref{thm:Dehn}, and assume for a contradiction that $q$ and $r$ are not commensurable. Then there exists a  linear function $f: V \to \bb{Q}$ such that $f(q)=a/2-1$ and $f(r)=-1$. Thus $f(\mu(X))= -2$, $f(\mu(Y))= 0$, and $f(\mu(A_i))f(\mu(B_i))=2-a/2$ for $i=1,2$, while $$\sum_{i=3}^j f(\mu(A_j))f(\mu(B_j)) = \sum_{i=3}^j\left(\frac{\mu(A_j)}{r}\right)^2 f^2(r)>  a-4.$$
Combining these estimates as in (\ref{e:Dehn}) we obtain
\begin{align*}
0 &= f(\mu(X))f(\mu(Y)) = \sum_{i=1}^{k}f(\mu(A_i))f(\mu(B_i)) \\ &= \sum_{i=1}^{2}f(\mu(A_i))f(\mu(B_i)) +
\sum_{i=3}^{j}f(\mu(A_i))f(\mu(B_i)) + \sum_{i=j+1}^{k}f^2(\mu(A_i)) \\ &> (4-a)+(a-4) =0,
\end{align*}
the desired contradiction.
\end{proof}

Let us now outline the proof of Theorem~\ref{thm:graph+}. As a first step (Claim~\ref{c:cycle} below) we prove that every cycle $C$ of $\Gamma'$ has length commensurable with $\pi$. We do so by examining the structure of shortest paths in $\Gamma$ between the pairs of points in~$C$. The paths joining the pairs of points which lie at distance greater than~$\pi$ in~$C$ must take ``shortcuts'' which we refer to as \emph{chords}. We construct a square tiling of a cylinder, where the squares correspond to the chords and apply Theorem~\ref{thm:Dehn}.

Next we prove in Claim~\ref{c:dumbbell} that certain paths of $\Gamma'$, called bars, also have length commensurable with $\pi$. We say that a path $B$ in $\Gamma'$ with endpoints $u$ and $v$ is a \emph{bar in $\Gamma'$ joining cycles $C_1$ and $C_2$}, if $C_1$ and $C_2$ are disjoint cycles of $\Gamma'$, and, moreover, $u \in C_1$, $v \in C_2$ and $C_1$ and $C_2$ are otherwise disjoint from $B$. See Figure~\ref{fig:chord}(b). Note that for every bar $B$ there exists a closed edge-path immersed in~$\Gamma'$ that traverses~$B$ twice and each of $C_i$ once. The existence of such an immersion allows us to adapt the argument we used for cycles to bars, but there are multiple technical hurdles to overcome, making this the most technical part of the proof. It is here, in particular, that we use Lemma~\ref{lem:Dehn+} and we additionally need extra information about the structure of chords of a cycle and between pairs of cycles, which we obtain in Claims~\ref{c:cycle} and~\ref{c:bicycle}.

To unify the argument for cycles and bars we start the proof of Theorem~\ref{thm:graph+} in the setting of Corollary~\ref{thm:graph} by considering closed edge-paths immersed in $\Gamma$.

Finally, in Claim~\ref{c:reduction}, we show that every segment  of $\Gamma'$ has length commensurable with $\pi$ by proving that it can be expressed as a rational linear combination of bars and cycles.

\begin{proof}[Proof of Theorem~\ref{thm:graph+}] We follow the steps outlined above, starting by considering immersed circles in $\Gamma'$.
	
	 Let $\gamma: C \to \Gamma'$ be a local isometry mapping a circle $C$ of length $l$ into $\Gamma'$.  We identify $C$ with $\R/l\Z$. This identification defines a natural $\R$-action on $C$.
	
	For $x,y \in C$, let $d(x,y)$ denote the distance between points $\gamma(x)$ and $\gamma(y)$ in $\Gamma$. As the girth of $\Gamma$ is at least $2\pi$, if  $0<d(x,y) < \pi$ then there exists a unique path from $\gamma(x)$ to $\gamma(y)$ in $\Gamma$ of length $d(x,y)$. We denote such a path by $P_{xy}$.
	
	Let $S$ be the set of points {\kol $s$} in $C$ such that $\gamma(s)$ has degree at least three in~$\Gamma$. For $s,t \in S$ we say that $(s,t)$ is a \emph{chord of $\gamma$ of length $d_0$} if $0<d_0=d(s,t) < \pi$, and $P_{st}$
is disjoint  from $\gamma([s-\eps,s) \cup (s,s+\eps] \cup [t-\eps,t) \cup (t,t+\eps])$ for some $\eps >0$. Note that our chords are directed, that is we distinguish between the chords $(s,t)$ and $(t,s)$.
	
		   \begin{figure}[h!]
		\begin{center}
			\includegraphics[scale=0.82]{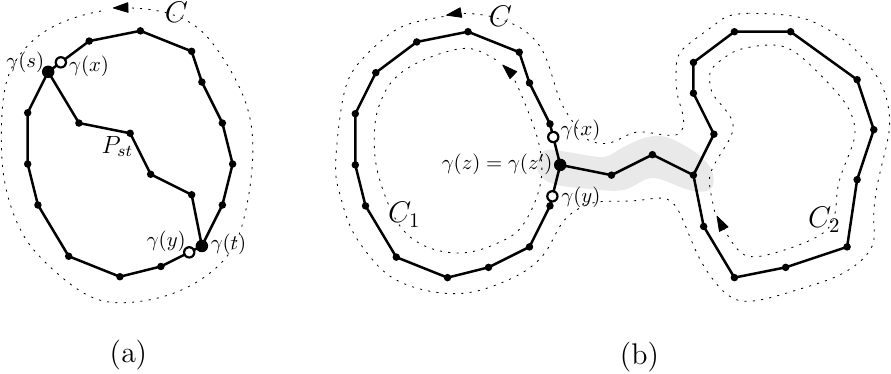}
		\end{center}
		\caption{(a) The pair $(x,y)$ uses the chord $(s,t)$. \newline			
			(b) A bar joining cycles $C_1$ and $C_2$ is shaded. The pair $(x,y)$
			is spliced in $C$. }
		\label{fig:chord}
	\end{figure}

	For $x,y \in C$ we say that $(x,y)$ \emph {uses a chord $(s,t)$} if there exists a path $P$ from $\gamma(x)$ to $\gamma(y)$ in $\Gamma$ of length $d(x,y)$ such that $P$ is obtained by concatenating paths $\gamma([x,s])$, $P_{st}$ and $\gamma([t,y])$, where we denote by $[a,b]$ for $a,b \in C$ the shortest of the two intervals in $C$ with endpoints $a$ and $b$. See Figure~\ref{fig:chord}(a). Clearly, $(x,y)$ uses a chord $(s,t)$ of length $d_0$ if and only if there exist $\sigma,\tau \in \R$ such that $x = s+\sigma$, $y =t + \tau$, and $|\sigma|+|\tau| \leq \pi - d_0$. Moreover, if the last inequality is strict, then $(s,t)$ is the unique chord used by $(x,y)$.
	
	Note further that for every $(x,y) \in C^2$ if $(x,y)$ uses no chord, then there exist $z,z' \in C$ such that $\gamma(z)=\gamma(z')$, and there exists a path from $\gamma(x)$ to $\gamma(y)$ in $\Gamma$ of length $d(x,y)$ obtained by concatenating paths $\gamma([x,z])$  and $\gamma([z',y])$. Moreover, if $d(x,y)<\pi$ then the converse also holds. If $(x,y)$ uses no chord and the distance between $x$ and $y$ in $C$ is larger than $\pi$ then we say that $(x,y)$ is a \emph{spliced pair}, see Figure~\ref{fig:chord}(b). We denote the set of spliced pairs by $Spl(\gamma)$.
	
	We now describe the square tiling to which we apply Theorem~\ref{thm:Dehn} and Lemma~\ref{lem:Dehn+}. Let $R=R(\gamma)\subset C^2$ be the annulus consisting of points $(x,y)$ with $x$ and $y$ at distance at least $\pi$ in $C$. Note that $R$ is isometric to the product of a circle of length $l\sqrt{2}$ and an interval of length $l\frac{\sqrt{2}}{2}-\pi\sqrt{2}$. Let $\mc{K}=\mc{K}(\gamma)$ denote the set of all chords, and let $S(K)$ for $K \in \mc{K}$ denote the set of pairs $(x,y) \in C^2$ which use~$K$. For a chord $K=(s,t)$ of length $d_0$, let $z=\pi - d_0$. As noted above, $$S(K) = \{(s+\sigma,t + \tau ) \: | \: -z \leq \sigma + \tau \leq z,  -z \leq \tau - \sigma \leq z \},$$ which is a square in $R$ with side length $z\sqrt 2$ and one side pair parallel to $\partial R$. Moreover, for distinct $K$ the squares $S(K)$ have disjoint interiors. Finally note that $Spl(\gamma)=\mathrm{int}(R)-\bigcup_{K \in \mc{K}}S(K)$ by definition.
	
	We will also use some of the terminology introduced above when considering a subgraph $\Lambda$ of $\Gamma$, rather than an immersion. For vertices $s,t$ of $\Lambda$ we say that $(s,t)$ is a \emph{chord of $\Lambda$ of length $d_0$ and area $2(\pi-d_0)^2$} if
	$0<d_0=d_{\Gamma}(s,t) < \pi$ and the first and last edge of the unique path $P$ from $s$ to $t$ in $\Gamma$ of length $d_0$ do not lie in~$\Lambda$.
Note that if $\gamma:C \to \Gamma$ is an immersion as above, and $s,t \in C$ are such that $(\gamma(s),\gamma(t))$ is a chord of $\gamma(C)$ of length $d_0$ then $(s,t)$ is a chord of $\gamma$ of length $d_0$. If $\gamma$ is injective then the converse also holds.
	
	We are now ready for the first major claim, which in particular establishes the theorem in the case when $\Gamma'$ is a cycle.
	
	\begin{claim}\label{c:cycle}
		Let $C$ be a cycle in $\Gamma'$ of length $l$. Then
		\begin{description}
			\item[a)] $l$ is commensurable with $\pi$,
			\item[b)] the length of every chord of $C$ is commensurable with $\pi$,
			\item[c)] for every point $p \in C$ the total area of chords of $C$ with both endpoints in $C-\{p\}$ is at least $2\pi(l-2\pi)$.
		\end{description}
	\end{claim}	
	
    \begin{proof} Let $\gamma$ be an isometry from a circle of length $l$ onto $C$. For simplicity of notation we identify the domain of $\gamma$ with $C$.
    	
    	Note that
 $Spl(\gamma)= \emptyset$. Thus  $\mathrm{int}(R(\gamma)) \subseteq \bigcup_{K \in \mc{K}(\gamma)}S(K) \subseteq R(\gamma)$. As $\mc{K}(\gamma)$ is finite, it follows that $R(\gamma) =\bigcup_{K \in \mc{K}(\gamma)}S(K) $, and Theorem~\ref{thm:Dehn} implies that   $l\sqrt{2}$ and $l\frac{\sqrt{2}}{2}-\pi\sqrt{2}$ are commensurable. Thus a) holds, and  Theorem~\ref{thm:Dehn} similarly implies that b) holds.
	
	It remains to establish c). First note that the total area of all chords of $C$ is equal to the area of $R(\gamma)$, which is $l(l-2\pi)$. It remains to upper bound the area of chords of $C$ with at least one end in $p$, i.e.\ chords of $\gamma$ with at least one end in $s = \gamma^{-1}(p)$. Let $K_1=(s,t_1),K_2=(s,t_2),\ldots,K_n=(s,t_n)$ be all the chords starting at $s$.
	Suppose that $K_i$ has length $\pi - z_i$. Then the area of $S(K_i)$ is equal to $2z_i^2 $, and $S(K_i)$ intersects the line $(\{s\} \times C) \cap R$ in an interval of length $2z_i$.  As $(\{s\} \times C) \cap R$ has length $l-2\pi$, we have $\sum_{i=1}^n z_i \leq l/2 -\pi$, and so the sum of the areas of the chords starting at $s$ is at most $2(\sum_{i=1}^n z_i)^2 \leq 2(l/2 -\pi)^2.$ It follows that the area of chords of $C$ with at least one end in $p$, is at most $(l-2\pi)^2$. As $l(l-2\pi) - (l-2\pi)^2 = 2\pi(l-2\pi)$, c) follows. \end{proof}
	
	Next we need to extend Claim~\ref{c:cycle} b) to a pair of disjoint cycles.
	
	\begin{claim}\label{c:bicycle}
		Let $C_1$, $C_2$ be disjoint cycles of $\Gamma'$. Then every chord of $C_1 \cup C_2$ has length commensurable with $\pi$.
	\end{claim}	

    \begin{proof} Let $l_i$ be the length of $C_i$ for $i=1,2$.
    	By Claim~\ref{c:cycle} a) and b) the lengths $l_1$ and $l_2$ are commensurable with $\pi$ and so is the length of every chord of $C_1$ and $C_2$. It remains to establish the claim for chords with one endpoint in $C_1$ and another in~$C_2$.
    	
    	To do so we parallel the proof of Claim~\ref{c:cycle} and the preceding construction of the square tiling. Let $\mc{K}$ be the set of chords $(s,t)$ of $C_1 \cup C_2$ with $s \in C_1$, $t \in C_2$. We say that a pair $(x,y) \in C_1 \times C_2$ \emph{uses a chord  $(s,t) \in \mc{K}$} if there exists a path of length $d_{\Gamma}(x,y)$ from $x$ to $y$ in $\Gamma$ obtained by concatenating a path in $C_1$, the minimal length path from $s$ to $t$, and a path in $C_2$. Let $S(K) \subseteq C_1 \times C_2$ be the set of pairs of points using  the chord $K$ of length $\pi-z$. Then analogously to the earlier observations we have that $S(K)$ is a square with side  length $\sqrt 2 z$, the squares $\{S(K)\}_{K \in \mc{K}}$ cover
    	$C_1 \times C_2$ and have disjoint interiors.

To apply Theorem~\ref{thm:Dehn} we need to transform the tiling  of $C_1 \times C_2$ into a tiling of a product set, where the squares inherit the product structure.
We do it as follows. Let $n_1,n_2$ be positive integers such that $n_1l_1=n_2l_2$. By lifting our tiling of $C_1 \times C_2$ to the product of the $n_1$-fold cover of $C_1$ and the $n_2$-fold cover of $C_2$ we may assume without loss of generality that $C_1$ and $C_2$ have the same length $l$.

  Let $\gamma_1: \bb{R} /l\bb{Z}\to C_1$ and $\gamma_2: \bb{R} /l\bb{Z}\to C_2$ be isometries. Let $\psi: \bb{R} /l\bb{Z}\times \bb{R} /l\bb{Z}\to  C_1 \times C_2$ be defined by $\psi(x,y)=(\gamma_1(x+y),\gamma_2(x-y))$. Consider a chord $K=(s,t) \in \mc{K}$ of length $\pi-z$. Then $\psi^{-1}(S(K))$ consists of two squares of the form $[x-z/2,x+z/2] \times [y-z/2,y+z/2]$, where $(x,y)\in (\bb{R}/l\bb{Z})^2$ is either one of the two pairs satisfying $\gamma_1(x+y)=s,\gamma_2(x{\kol-}y)=t$. Thus the preimage of our tiling of $C_1 \times C_2$ is a square tiling of  $(\bb{R}/l\bb{Z})^2$ satisfying the conditions of Theorem~\ref{thm:Dehn}, and by this theorem if $K \in \mc{K}$ is a chord of length $\pi-z$, then $z$ is commensurable with~$l$. As $l$ is commensurable with $\pi$, the claim follows.
  \end{proof}

We are now ready to establish that the conclusion of the theorem holds for bars.
	
	\begin{claim}\label{c:dumbbell}
		The length of every bar of $\Gamma'$ is commensurable with $\pi$.
		
	\end{claim}
	
	\begin{proof}
		Let $B$ be a bar in  $\Gamma'$ with endpoints $u$ and $v$, of length $b$, joining cycles $C_1$ and $C_2$ of length $l_1$ and $l_2$ respectively such that $u \in C_1$, $v \in C_2$. Let $\gamma: C \to \Gamma'$ be a local isometry from a circle $C$ of length $l=l_1+l_2+2b$, traversing $B$ twice, and each of $C_1$ and $C_2$ once.
		Our goal is to apply Lemma~\ref{lem:Dehn+} to a tiling of $R(\gamma)$ defined in the beginning of the proof.
		
		 Unlike in Claim~\ref{c:cycle}, the set $Spl(\gamma)$ of spliced pairs is not empty, but it is not hard to analyze. Let $\gamma^{-1}(B)=[s_1,s_1+b] \cup [s_2-b,s_2]$ for some $s_1,s_2 \in C$.
		  Then $Spl(\gamma) $ consists of two rectangles: one with corners $(s_1-\pi,s_2),(s_1,s_2+\pi),(s_1+b+\pi,s_2-b),(s_1+b,s_2-b-\pi)$, and the other obtained  from it by permuting the coordinates.

		  Let $r =\pi \sqrt 2$, and let $q = b \sqrt 2$. Let $a=(l_1+l_2)/\pi$. Then $l\sqrt 2 = 2q+ar$, and thus $R(\gamma)$ can be considered as a product of a circle $\bd R$ of length $2q+ar$ and an interval of length $q+(a/2-1)r$.
		  Each of rectangles in $Spl(\gamma)$ has a side of length $r$ parallel to $\bd R(\gamma)$, and a side of length $q+r$ orthogonal to it.  The cylinder $R(\gamma)$ is tiled by the rectangles in $Spl(\gamma)$ and squares in $\{S(K)\}_{K \in \mc{K}(\gamma)}$, and the first three conditions of Lemma~\ref{lem:Dehn+} are satisfied for this tiling. Note that Lemma~\ref{lem:Dehn+} would immediately imply the claim.
		
		  Thus to prove the claim it remains to verify the last condition of Lemma~\ref{lem:Dehn+}, i.e.\ to show that the total area of the squares corresponding to the chords of $\gamma$ with length commensurable with $\pi$ is strictly greater than $(a-4)r^2 = 2\pi(l_1+l_2)-8\pi^2$.
		
		  By Claim~\ref{c:bicycle}, every chord of $C_1 \cup C_2$ has length commensurable with $\pi$, and if such a chord has no endpoint in $\{u,v\}$ then it corresponds to a chord of $\gamma$ of the same length.
		  By Claim~\ref{c:cycle} c) the total area of the chords of $C_i$ with no endpoint in $\{u,v\}$ is at least  $2\pi l_i-4\pi^2$.
	It remains to find at least one chord of $C_1 \cup C_2$ with one endpoint in $C_1 - \{u\}$ and another in $C_2 - \{v\}$. Indeed, consider a point $x \in C_1$ at distance at least~$\pi$ from $u$ in $C_1$, and a point $y \in C_2$ at distance at least~$\pi$ from~$u$ in $C_2$. Then the chord used by $(x,y)$ cannot have an endpoint in $\{u,v\}$ and so is as required.
	  		\end{proof}

Let $V(\Gamma'), E(\Gamma')$ denote the vertex and edge set of the graph $\Gamma'$.	
We finish the proof of the theorem by reducing it to Claims~\ref{c:cycle} a) and~\ref{c:dumbbell}. For the reduction it will be convenient for us to think of subgraphs of $\Gamma'$ as elements of $\bb{Q}^{E(\Gamma')}$, the vector space of formal linear combinations of edges of $\Gamma'$ with rational coefficients. Thus we identify every subgraph $\Lambda$ of $\Gamma'$ with $\sum_{e \in E(\Lambda)}e$. The theorem immediately follows from Claims~\ref{c:cycle} a) and~\ref{c:dumbbell} and the next claim, which uses the above convention.

\begin{claim}\label{c:reduction}
Every segment of $\Gamma'$ is a rational linear combination of cycles and bars of $\Gamma'$.
\end{claim}
	
\vskip10pt \noindent \emph{Proof.}
	Let $P$ be a segment of $\Gamma'$. We assume without loss of generality that $\Gamma'$ is chosen minimal subject to the conditions that the minimum degree of $\Gamma'$ is at least two, and $P$ is a segment of $\Gamma'$. We further assume by suppressing vertices of degree two in $\Gamma'$ that every vertex of $\Gamma'$ has degree at least three. In particular, every segment of $\Gamma'$ is an edge. Subject to these assumptions there are only a few isomorphism types of graphs to consider, and the proof proceeds by exhaustive case analysis.
	
	Note first that every non-loop edge $e \in E(\Gamma')$ shares at least one endpoint with $P$, as otherwise $\Gamma' - e$ contradicts the minimality of $\Gamma'$. Similarly, if $w \in V(\Gamma') -V(P)$ is incident to a loop then $\deg(w)=3$. (We use $\deg(x)$ to denote the degree of a vertex $x$ in $\Gamma'$.)  By the previous observation the non-loop edge incident to $w$ shares an endpoint with $P$. We call such a vertex $w$ a \emph{pseudoleaf} of $\Gamma'$.
	
	Let $u$ and $v$ be the endpoints of $P$. Suppose first that $\deg(u) \geq 4$. Let $e \in E(\Gamma')-E(P)$ be any edge incident to $u$. By the minimality of $\Gamma'$, either $u$ or $v$ has degree at most two in $\Gamma'- e$. Thus either $e$ is a loop and $\deg(u) = 4$, or $e$ joins $u$ and $v$ and $\deg(v) = 3$. As one of these outcomes holds for every edge in $E(\Gamma')-E(P)$ incident to $u$, we conclude that $\Gamma'$ consists of $P$, a loop incident to~$u$, an edge parallel to $P$, a pseudoleaf adjacent to $v$, which we denote by $w$, and a loop incident to $w$, see Figure~\ref{fig:Claim4}(a).
	
   \begin{figure}[h!]
	\begin{center}
		\includegraphics[scale=0.6]{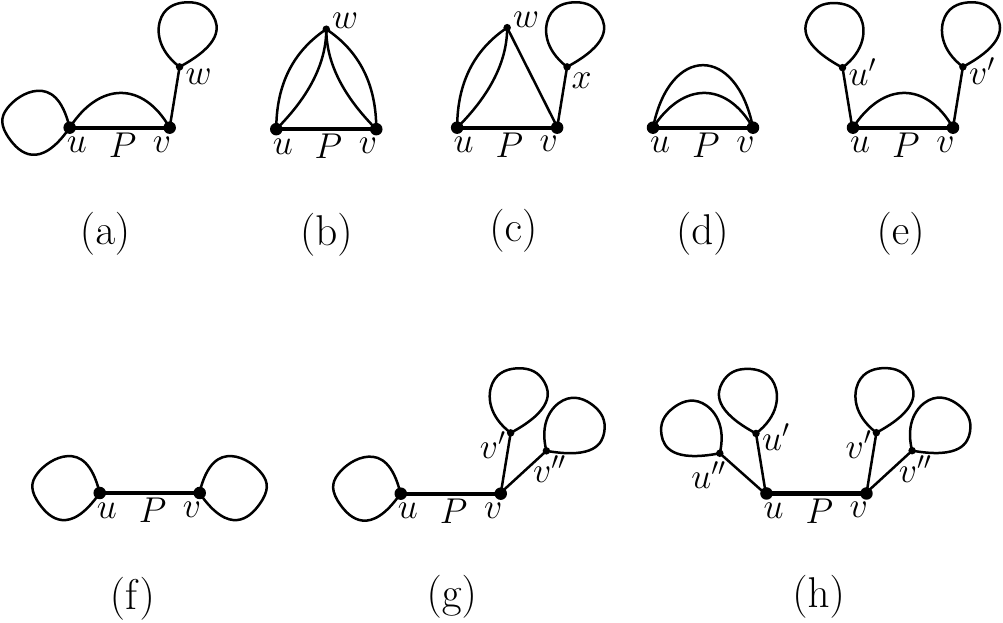}
	\end{center}
	\caption{Cases in the proof of Claim~\ref{c:reduction}.}
	\label{fig:Claim4}
\end{figure}

	It is not hard to verify that the claim holds in this case, and we do so using the following notation, which will also be used in all the remaining cases.  We denote the edge with endpoints $x$ and $y$ by $xy$, and in the case when there are several such edges in $\Gamma'$ we denote them by $xy_1$, $xy_2$,\ldots. In particular, let $P=uv_1$.
	
	Returning to our case, note that $wv$ is a bar joining a loop at $w$ and a cycle $P + uv_2$, while $wv+P$ is a bar joining loops at $w$ and $u$. Thus $P=(wv+P)-wv$ is a difference of two bars, and the claim holds in this case. The case $\deg(v) \geq 4$ is symmetric.

    It remains to consider the case $\deg(u)=\deg(v)=3$. Suppose next that there exists $w \in V(\Gamma')-\{u,v\}$ such that $w$ is not a pseudoleaf. As every non pseudoleaf vertex in  $V(\Gamma')-\{u,v\}$ is incident to at least three edges which have $u$ or $v$ as their second endpoint, every vertex in $V(\Gamma')$ except $u,v$ and $w$ is a pseudoleaf.

    If $\deg(w)\geq 4$, then $w$ is joined to each of $u$ and $v$ by a pair of parallel edges, see Figure~\ref{fig:Claim4}(b).
    In this case each of $wv_1+wv_2$, $wu_1+wu_2$, $wv_1+wu_1+P$, $wv_2+wu_2+P$ is a cycle, and
    $P$ is a rational linear combination of these cycles. Thus the claim holds.

    If $\deg(w)=3$, we suppose without loss of generality that $w$ is joined to $u$ by a pair of parallel edges, and to $v$ by an edge. Then the remaining edge incident to $v$ must have a pseudoleaf, which we denote by $x$, as its second endpoint, see Figure~\ref{fig:Claim4}(c). In this case, $xv$ is a bar joining the loop at $x$ to the cycle $P+uw_1+wv$, and $xv+P$ is a bar joining the loop at $x$ and the cycle $uw_1+uw_2$. Therefore, $P$ is the difference of two bars, and the claim holds.

    We reduced our analysis to the case when $\deg(u)=\deg(v)=3$ and every vertex in $V(\Gamma')-\{u,v\}$ is a pseudoleaf. We now consider subcases depending on the number of edges joining $u$ to $v$.

    If there are three such edges, then $P+uv_2$, $P+uv_3$ and $uv_2+uv_3$ are each cycles, and $P$ is a rational linear combination of these cycles, see Figure~\ref{fig:Claim4}(d). If exactly two edges join $u$ to $v$, then there is a pseudoleaf adjacent to each of $u$ and~$v$. Denote these pseudoleaves by $u'$ and $v'$ respectively, see Figure~\ref{fig:Claim4}(e). Then $uu'$ is a bar joining the loop at $u'$ and $P+uv_2$, and similarly $vv'$ is a bar. The path $uu'+P+vv'$ is also a bar, joining the loops at $u'$ and $v'$, and $P$ is a rational linear combination of the above three bars.

    It remains to consider the case when $P$ is the unique edge between $u$ and $v$. In this case each of $u$ and $v$ is either incident to a loop, or adjacent to two pseudoleaves. Up to symmetry there are three final cases to consider.

    If both $u$ and $v$ are incident to a loop then $P$ is a bar, see Figure~\ref{fig:Claim4}(f). If $u$ is incident to a loop and $v$ is adjacent to pseudoleaves $v'$ and $v''$, then each of the paths $v'v+P$, $v''v+P$ and $v'v+v''v$ is a bar joining the loops at its endpoints, and $P$ is a rational linear combination of these bars, see Figure~\ref{fig:Claim4}(g). Finally, suppose that $v$ is adjacent to pseudoleaves $v'$ and $v''$, and $u$ is adjacent to pseudoleaves $u'$ and $u''$, see Figure~\ref{fig:Claim4}(h). Then $P$ is a rational combination of bars $v'v+v''v$, $u'u+u''u$, $v'v+P + u'u$, $v''v+P+u''u$. Thus the claim holds in this last case.
 \end{proof}


\begin{bibdiv}
\begin{biblist}

\bib{AZ}{book}{
	title={Proofs from the Book},
	author={Aigner, Martin}
	author={Ziegler, G{\"u}nter M.},
	volume={274},
	year={2018},
	publisher={Springer}
}

\bib{BB}{article}{
   author={Ballmann, Werner},
   author={Brin, Michael},
   title={Orbihedra of nonpositive curvature},
   journal={Inst. Hautes \'{E}tudes Sci. Publ. Math.},
   number={82},
   date={1995},
   pages={169--209 (1996)}}

\bib{Babu}{article}{
   author={Ballmann, Werner},
   author={Buyalo, Sergei},
   title={Nonpositively curved metrics on $2$-polyhedra},
   journal={Math. Z.},
   volume={222},
   date={1996},
   number={1},
   pages={97--134}}

\bib{BestvinaProblemOLD}{article}{
	AUTHOR = {Bestvina, Mladen},
	TITLE = {Questions in Geometric Group Theory},
	date = {2000},
	eprint = {https://www.math.utah.edu/~bestvina/eprints/questions.pdf}
}

\bib{BridsonProblem}{article}{
	AUTHOR = {Bridson, Martin R.},
	TITLE = {Problems concerning hyperbolic and ${\rm CAT}(0)$ groups},
	date = {2007},
	eprint = {https://docs.google.com/file/d/0B-tup63120-GVVZqNFlTcEJmMmc/edit}
}

\bib{BH}{book}{
   author={Bridson, Martin R.},
   author={Haefliger, Andr\'e},
   title={Metric spaces of non-positive curvature},
   series={Grundlehren der Mathematischen Wissenschaften [Fundamental
   Principles of Mathematical Sciences]},
   volume={319},
   publisher={Springer-Verlag},
   place={Berlin},
   date={1999}
   }

\bib{Cap}{article}{
	author={Caprace, Pierre-Emmanuel},
	title={Lectures on proper $\rm CAT(0)$ spaces and their isometry groups},
	conference={
		title={Geometric group theory},
	},
	book={
		series={IAS/Park City Math. Ser.},
		volume={21},
		publisher={Amer. Math. Soc., Providence, RI},
	},
	date={2014},
	pages={91--125}}

\bib{CL}{article}{
   author={Caprace, Pierre-Emmanuel},
   author={Lytchak, Alexander},
   title={At infinity of finite-dimensional ${\rm CAT}(0)$ spaces},
   journal={Math. Ann.},
   volume={346},
   date={2010},
   number={1},
   pages={1--21}}


\bib{Dehn}{article} {
	title = {{\"U}ber zerlegung von rechtecken in rechtecke},
	author = {Dehn, Max},
	journal = {Math. Ann.},
	volume = {57},
	number = {3},
	pages = {314--332},
	year = {1903},
}

\bib{Gri}{article}{
	AUTHOR = {Grigorchuk, Rostislav I.},
	TITLE = {Degrees of growth of finitely generated groups and the theory
		of invariant means},
	JOURNAL = {Izv. Akad. Nauk SSSR Ser. Mat.},
	VOLUME = {48},
	YEAR = {1984},
	NUMBER = {5},
	PAGES = {939--985},
}

\bib{Hadwiger}{book}{
	title={Vorlesungen {\"u}ber Inhalt, Oberfl{\"a}che und Isoperimetrie},
	author={Hadwiger, Hugo},
	year={1957}
	publisher={Springer}
}

\bib{HMP}{article}{
	author={Hanlon, Richard Gaelan},
	author={Mart\'{i}nez-Pedroza, Eduardo},
	title={Lifting group actions, equivariant towers and subgroups of
		non-positively curved groups},
	journal={Algebr. Geom. Topol.},
	volume={14},
	date={2014},
	number={5},
	pages={2783--2808},
}

\bib{H}{book}{
   author={Hatcher, Allen},
   title={Algebraic topology},
   publisher={Cambridge University Press, Cambridge},
   date={2002},
   pages={xii+544}}

\bib{I}{article}{
   author={Ivanov, Sergei},
   title={On Helly's theorem in geodesic spaces},
   journal={Electron. Res. Announc. Math. Sci.},
   volume={21},
   date={2014},
   pages={109--112}}

\bib{KapovichProblems}{article}{
	AUTHOR = {Kapovich, Michael},
	TITLE = {Problems on boundaries of groups and Kleinian groups},
	date = {2008},
	eprint = {https://docs.google.com/file/d/0B-tup63120-GOC15RW8zMDFZZjg/edit}
}

\bib{LP}{article}{
   author={Lamy, St\'ephane},
   author={Przytycki, Piotr},
   title={Presqu'un immeuble pour le groupe des automorphismes mod\'er\'es},
   journal={Ann. H. Lebesgue},
   status={to appear},
   eprint={arXiv:1802.00481},
   date={2020}}

\bib{Mar}{article}{
	author={Marquis, Timoth\'{e}e},
	title={A fixed point theorem for Lie groups acting on buildings and
		applications to Kac-Moody theory},
	journal={Forum Math.},
	volume={27},
	date={2015},
	number={1},
	pages={449--466}}

\bib{M}{article}{
   author={Masur, Howard},
   title={Closed trajectories for quadratic differentials with an
   application to billiards},
   journal={Duke Math. J.},
   volume={53},
   date={1986},
   number={2},
   pages={307--314}}

\bib{Osa}{article}{
	author={Osajda, Damian},
	title={Group cubization},
	note={With an appendix by Mika\"{e}l Pichot},
	journal={Duke Math. J.},
	volume={167},
	date={2018},
	number={6},
	pages={1049--1055}}

\bib{PapSwe}{article}{
	author={Papasoglu, Panos},
	author={Swenson, Eric},
	title={Boundaries and JSJ decompositions of ${\rm CAT}(0)$-groups},
	journal={Geom. Funct. Anal.},
	volume={19},
	date={2009},
	number={2},
	pages={559--590}}

\bib{PapSwe2}{article}{
	author={Papasoglu, Panos},
	author={Swenson, Eric},
	title={Finite cuts and ${\rm CAT}(0)$ boundaries},
	eprint={arXiv:1807.04086},
	status={sumbitted},
	date={2018}}

\bib{Parreau}{article}{
   author={Parreau, Anne},
   title={Sous-groupes elliptiques de groupes lin\'{e}aires sur un corps valu\'{e}},
   language={French, with English summary},
   journal={J. Lie Theory},
   volume={13},
   date={2003},
   number={1},
   pages={271--278}}

\bib{Sag}{article}{
	author={Sageev, Michah},
	title={Ends of group pairs and non-positively curved cube complexes},
	journal={Proc. London Math. Soc. (3)},
	volume={71},
	date={1995},
	number={3},
	pages={585--617}}

\bib{Swe}{article}{
	author={Swenson, Eric L.},
	title={A cut point theorem for ${\rm CAT}(0)$ groups},
	journal={J. Differential Geom.},
	volume={53},
	date={1999},
	number={2},
	pages={327--358}}

\bib{Z}{article}{
   author={Zeeman, E. Christopher},
   title={Relative simplicial approximation},
   journal={Proc. Cambridge Philos. Soc.},
   volume={60},
   date={1964},
   pages={39--43}}

\end{biblist}
\end{bibdiv}
\end{document}